\documentclass[a4paper,10pt,reqno]{amsart}
\usepackage{amssymb,latexsym,amsmath,amsfonts}
\usepackage{amsthm}
\usepackage[mathscr]{eucal}
\usepackage{rotating}
\usepackage{multirow}

\numberwithin{equation}{section}
\newtheorem{theorem}{Theorem}
\newtheorem{lemma}{Lemma}

\theoremstyle{definition}
\newtheorem{definition}[theorem]{Definition}

\newtheorem{example}[theorem]{Example}

\theoremstyle{remark}

\newcommand {\sgn}{\mbox{sgn}}

\begin{document}

\title[Unboundedness Properties in Volterra Summation Equations]
{
On the admissibility of unboundedness properties of forced  deterministic and stochastic sublinear Volterra summation equations}

\author{John A. D. Appleby}
\address{School of Mathematical
Sciences, Dublin City University, Glasnevin, Dublin 9, Ireland}
\email{john.appleby@dcu.ie} 

\author{Denis D. Patterson}
\address{School of Mathematical
Sciences, Dublin City University, Glasnevin, Dublin 9, Ireland}
\email{denis.patterson2@mail.dcu.ie}

\thanks{John Appleby gratefully acknowledges the support of a SQuaRE activity entitled ``Stochastic stabilisation of limit-cycle dynamics in ecology and neuroscience'' funded by the American Institute of Mathematics. Denis Patterson is supported by the Irish Research Council grant GOIPG/2013/402 ``Persistent and strong dependence in growth rates of solutions of stochastic and deterministic functional differential equations with applications to finance''.
} 
\date{1 July 2016}

\begin{abstract}
In this paper we consider unbounded solutions of perturbed convolution Volterra summation equations. The equations studied are asymptotically sublinear, in the sense that the state--dependence in the summation is of smaller than linear order for large absolute values of the state. When the perturbation term is unbounded, it is elementary to show that solutions are also. The main results of the paper are mostly of the following form: the solution has an additional unboundedness property $U$ if and only 
if the perturbation has property $U$. Examples of property $U$ include monotone growth, monotone growth with fluctuation, fluctuation on $\mathbb{R}$ without growth, existence of time averages. We also study the connection between the times at which the perturbation and solution reach their running maximum, and the connection between the size of signed and unsigned running maxima of the solution and forcing term.
\end{abstract}

\maketitle

\section{Introduction}
In this paper we determine conditions under which the solutions of a forced Volterra summation equation 
of the form
\begin{equation} \label{eq.xintro}
x(n+1)=H(n+1)+\sum^{n}_{j=0}k(n-j)f(x(j)),\quad n\geq0.
\end{equation}
have bounded and unbounded solutions. It is assumed that $f(x)=o(x)$ as $|x|\to\infty$ and $k$ is summable. These properties of $f$ and $k$ ensures the boundedness of the Volterra equation under moderate disturbances in the system: in fact, we show that $x$ is unbounded 
if and only if the external force $H$ is. Once that has been done, the bulk of the paper is devoted to exploring refinements of these unboundedness results. Generally speaking, we find that if $H$ has an additional unboundedness property $U$, then $x$ inherits property $U$. In many cases, the converse is also true.  

In this sense, our results are related to the theory of admissibility for Volterra summation equations 
(cf. Baker and Song~\cite{bakersong04,bakersong06}, Reynolds~\cite{rey}, Reynolds and Gy\H{o}ri~\cite{gyrey09,gyrey10}, 
Gy\H{o}ri, Horv\'ath~\cite{gyhorv08,gyhorv08b} and Awwad and Gy\H{o}ri~\cite{gyaw12}) and Volterra integral equations (see e.g., 
~\cite{ja2010,jagydr11} inspired by work of Perron~\cite{perron} and Corduneanu~\cite{cord1973}). In many of the discrete papers the theme of research is often (but certainly not exclusively) to consider the following problem 
\[
x(n)=H(n)+(Vx)(n), \quad n\geq 0
\] 
where $V$ is a linear Volterra operator and $H$ is an $\mathbb{R}^p$--valued sequence. The solution $x$ is a sequence in $\mathbb{R}^p$. The form of $V$ is 
\[
(Vg)(n):=\sum_{j=0}^n K(n,j)g(j)=:(K\star g)(n), \quad n\geq 0
\]
for any $g:\mathbb{Z}^+\to \mathbb{R}^p$, where $K$ is fixed and $K:\mathbb{Z}^+\times \mathbb{Z}^+ \to \mathbb{R}^{p\times p}$. If the operator $V$ (or equivalently, the matrix $K$) has the appropriate properties, and $H$ is a sequence with a nice asymptotic property characterised by a sequence space $N$, $x$ will lie in the space $N$. 
In such situations, it will be the case that $g\mapsto K\star g$ takes $N$ to $N$ and we say that the mapping has an admissibility property. Sometimes, it can even happen that properties of $x$ may be enhanced.  

Much effort has gone into investigating nice spaces $N$, such as bounded, convergent, periodic, or $\ell^q$ spaces of sequences.     
Instead, the results in this paper have a rather different flavour, as the types of external force $H$ are either highly irregular (for example stochastic) or are unbounded or growing. One consequence of this is that it becomes reasonable to track new types of property, such as the the size of the largest fluctuations to date (both positive, negative, and in absolute terms), the times at which sequences $H$ and $x$ reach their maximum value to date, growth rates and growth bounds, or indeed time averages of functions of the sequences 
(which may be finite even though those sequences are unbounded). Therefore, in a sense, our results are of greater applicability in economics or finance, rather than engineering, because disturbances to the system are less likely to be regular or bounded in applications in the former disciplines, while in engineering, we cannot expect good performance if disturbances are irregular or unbounded. 

Incidentally, we note if $f$ is linear, and $H$ is in the class of stationary ARMA processes, then $x$ is a stationary 
process provided the equation without noise is stable. The statistical behaviour of such linear models is well known, and therefore is expressly not the subject of this work; however, path properties are less well understood, and in a parallel work we explore the properties of \eqref{eq.xintro} using the same framework as in the present paper. Once again, we observe the pattern of this paper that the unboundedness property $U$ of the external force $H$ transmits to the solution $x$, and indeed, owing to the linearity of the problem the connection with the above--cited works on admissibility is more tangible: indeed, to a certain degree our contribution in the linear case is merely to identify nice spaces, and then to check by direct calculation, or by appeal to the general theory, that admissibility properties hold. In this work and its linear counterpart, we have focused on the convolution equation with a view towards applications and statistical inference: in this we merely adopt the time--honoured \emph{ceteris paribus} perspective of the economist in trying to keep the structure of the model time--independent (apart from possible external shocks), and the desire of the statistician to consider models with time--invariant statistical properties. These constraints, which are purely driven by applications, lead us to study the (asymptotically) autonomous convolution operator in \eqref{eq.xintro}. However, nothing prevents the results in this paper, as well as the convolutional  linear case, being developed for non--convolution Volterra equations, nor indeed the extension of the analysis to deal with the general $p$--dimensional case. Notwithstanding this, we feel that a substantial challenge has already been met by analysing
successfully the scalar case.      

Some of our results require the sequence $H$ to be stochastic, but not all do. However, most of our results are inspired by a unspoken assumption that $H$ could be stochastic, and that interesting properties of random processes can be assumed about $H$.
For this reason, we sometimes talk about $H$ as though it could be stochastic, and motivate our results by appealing to 
intuition about stochastic processes: therefore we freely use terminology like ``shock'', ``noise'' and ``stochastic process''
when talking about $H$. In our precise mathematical results, though, $H$ can be irregular or unbounded, but deterministic (i.e., $H$ could be a chaotic sequence), and our arguments would still be valid. In this sense, our analysis asks how the system modelled by the Volterra equation adjusts to shocks with certain characteristics, irrespective of whether they are stochastic or not.  

We first show that $H$ is unbounded, the maxima of $|H|$ and $|x|$ grow at the same rate, so that
\[
\lim_{n\to\infty} \frac{\max_{0\leq j\leq n} |x(j)|}{\max_{1\leq j\leq n} |H(j)|}=1,
\] 
This shows that shocks to the system, or growth from an external source, are not amplified nor damped by the system. However, it does not yet show whether unbounded fluctuations or growth in $H$ give rise to fluctuations or growth in $x$, but merely that the absolute size of the running maxima grow at the same rate. We also prove that the largest absolute fluctuations in $H$ to-date cause those in $x$.  
We do this by studying the the times $t_{n}^x$ and $t_{n}^H$ at which of the largest absolute fluctuations in $x$ and $H$ up to time $n$ occur, and show for example that $|x(t_n^x)|/|H(t_n^x)|\to 1$ and $|H(t_n^x)|/|H(t_n^H)|\to 1$ as $n\to\infty$. Our analysis to prove these results, and almost all others, is embarrassingly elementary and hinges mostly on careful analysis of the running maximum of sequences. Indeed, it is not unreasonable to state that almost all the analysis involves little more than taking maxima on both sides of \eqref{eq.xintro}, or obvious rearrangements of \eqref{eq.xintro} and parts thereof.  

It should be noted that the unboundedness of the sequences as described by these results does not make an assumption about whether $H$ grows or fluctuates. We show that essentially monotone growth in the forcing term $H$ produces monotone growth in $x$, and that $x$ is asymptotic to $H$; if the growth in $H$ is non--monotone, but a monotone trend about which $H$ grows can be identified, $x$ inherits this property also. 
We also study what happens when $H$ has large positive or large negative fluctuations. 
The main result shows that the dominating large fluctuations (positive or negative) in $H$ produce large positive or large negative fluctuations in $x$ of the same order of magnitude as those in $H$. It is also shown that when the large positive and negative fluctuations of $H$ are of the same order of magnitude, then $x$ has both large positive and negative fluctuations, and these follow the asymptotic growth of the respective fluctuations in $H$.  


Our first main results about absolute fluctuations show that 
\[
\max_{1\leq j\leq n} |x(j)|\sim \max_{1\leq j\leq n} |H(j)|, \quad \text{as $n\to\infty$}, 
\]
so in order to understand the growth in the partial maximum of $x$, it is necessary to determine the 
growth rate of $\max_{j\leq n} |H(j)|$. However, it is more straightforward, especially in the case when $H$ is an independent and identically distributed (iid) sequence, to try to find a deterministic sequence $a(n)$ which is increasing at the same rate as $H$ in absolute value terms. This is true because the Borel--Cantelli lemmas will only yield upper bounds for the growth in the maximum, while they can give upper and lower bounds in the iid case when the auxiliary sequence $a$ is introduced. 
In order to be of greater use for stochastic systems, we therefore prove that, for general sequences if 
$\limsup_{n\to \infty}|H(n)|/a(n)=\rho\in[0,\infty]$ then $\limsup_{n\to\infty}|x(n)|/a(n)=\rho$. This general result does not employ stochastic arguments.  

The final results examine the boundedness of time averages of the same function $\varphi$ of $H$ and $x$ and how they are related even though the sequences $H$ (and therefore $x$) are tacitly assumed to be unbounded (it is trivially the case that time averages of any well--behaved function of $H$ and $x$ will be finite if both sequences are bounded). In particular, we 
show the equivalence of these ``$\varphi$-moments'' of $H$ and $x$ in the case where $\varphi$ is an increasing and convex function. This covers important examples such as the finiteness of time averages, variances, skewness, and kurtosis for example (by taking $\varphi(x)=x^p$ for $p=1,2,3,4$). However, for ``thin tailed'' distributions, such as Gaussian distributions, we can consider non--power convex functions. The parameterised family $\varphi(x)=e^{ax^2}$ for $a>0$ is useful in the Gaussian case, for instance.    

\section{Mathematical Preliminaries}
\subsection{Notation and assumptions on data}
We now give the equation we study and impose hypotheses on the data. Suppose 
\begin{equation} \label{eq.f1}
f\in C(\mathbb{R};\mathbb{R}) 
\end{equation}
with 
\begin{equation} \label{eq.f2}
\lim_{|x|\to\infty}\frac{f(x)}{x}=0, 
\end{equation} 
and that $k=(k(n))_{n\geq 0}$ is a sequence with 
\begin{equation} \label{eq.kl1}
k\in l^1(\mathbb{N}).
\end{equation}
We find it useful to define
\begin{equation} \label{eq.k1}
\left|k\right|_1:=\sum^{\infty}_{j=0}\left|k(j)\right|<+\infty. 
\end{equation}
Let $(H(n))_{n\geq1}$ be a real sequence and let $(x(n))_{n\geq0}$ be another real sequence uniquely defined by
 \begin{align} \label{eq.x} 
x(n+1)=H(n+1)+\sum^{n}_{j=0}k(n-j)f(x(j)),\qquad n\geq0, \quad  x(0)=\xi\in \mathbb{R}.
\end{align}
We introduce the notation
\begin{align} \label{eq.Hstar} 
H^*(n):&=\max_{1\leq j \leq n}|H(j)|.
\end{align}
We also define
\begin{align} \label{eq.xstar} 
x^*(n):&=\max_{ 0 \leq j \leq n }|x(j)|.
\end{align}
The times to date at which these sequences reach their running maximums are also of interest, are denoted by
$t_{n}^x$, $t_{n}^H$ $\in{0,..,n}$, and defined by 
\begin{align} \label{eq.tstarnx} 
|x(t_{n}^x)|&=\max_{ 0 \leq j \leq n }|x(j)| \\
\label{eq.tstarnH}
|H(t_{n}^H)|&=\max_{ 1 \leq j \leq n }|H(j)|. 
\end{align}

We are generally interested in the behaviour of the solution $x$ of \eqref{eq.x} when $x$ becomes large (in absolute value terms), 
as when this happens, the solution is undergoing a large fluctuation of growth. We assume that $f$ is nonlinear, and are therefore particularly interested in the behaviour of $f(x)$ for large $|x|$. 
We assume that the impact of the past is of smaller that linear order for large $x$; the other extreme would be to consider when $f(x)/x\to\pm\infty$ as $x\to\pm\infty$, which we do not do here. The assumption that $f$ is sublinear in the sense of \eqref{eq.f2} achieves this. One of the effects of this assumption 
\eqref{eq.f2} is that the equation \eqref{eq.x} will be quite stable with respect to moderate disturbances. This is attractive, because this is  not always the case if $f$ is linear or obeys $f(x)/x\to\infty$ as $x\to\infty$. 

\subsection{Time indexing in the Volterra equation}
Before listing and discussing the main results, we stop to comment on the time indexing used in \eqref{eq.x}. Many authors choose to write 
\begin{equation} \label{eq.x2}
x(n+1)=H(n)+\sum_{j=0}^n k(n-j)f(x(j)), \quad n\geq 0
\end{equation}
or even study the equation 
\begin{equation} \label{eq.x3}
x(n)=H(n)+\sum_{j=0}^n k(n-j)f(x(j)), \quad n\geq 0
\end{equation}
(especially in the case that $f(x)=x$, and impose a solvability condition in order to ensure the existence of a solution of \eqref{eq.x3}). We prefer to express the equation in the form \eqref{eq.x}, however, for a technical reason related to the situation when $H$ is a stochastic process, and also from the perspective of viewing \eqref{eq.x} as modelling an economic system in which agents can observe the state of the system $x$ up to the current time $n$, but cannot know the future values of the system $\{x(j):j\geq n+1\}$ with certainty, owing to the randomness in $H$.

The appropriate probabilistic formulation of \eqref{eq.x} in the case that $H$ is a stochastic process is the following, and we will adopt this formulation. Suppose that $(\Omega, \mathcal{F}, (\mathcal{F}(n))_{n\geq 0},\mathbb{P})$ is an extended probability triple. We suppose that $(\mathcal{F}(n))_{n\geq 0}$ is a filtration (an increasing sequence of $\sigma$--algebras) with $\mathcal{F}\supset \mathcal{F}(n)$ for each $n\geq 0$ and $\mathcal{F}(n+1)\supset \mathcal{F}(n)$ for $n\geq 0$. We suppose that $H(n)$ is $\mathcal{F}(n)$--measurable for $n\geq 1$; the process $H$ is then said to be adapted to the filtration $(\mathcal{F}(n))_{n\geq 0}$, or adapted for short. 

Remembering that $\mathcal{F}(n)$ represents the information available about the system at time $n$, 
and granted the assumption \eqref{eq.f1} that $f$ is continuous and $k$ is deterministic, then in equation \eqref{eq.x},  
$x(n)$ is $\mathcal{F}(n)$--measurable for each $n\geq 1$, so $x$ is adapted. Therefore, the value of $x(n+1)$ is not known with certainty at time $n$, but is at time $n+1$, as soon as $H(n+1)$ has been observed. In the formulation \eqref{eq.x2}, however, if we still suppose that $H(n)$ is $\mathcal{F}(n)$--measurable, then $x(n+1)$ is known with certainty at time $n$. A process with this property is called previsible or predictable, and typically we would not wish to assume a priori in a discrete--time economic model that a publicly visible state of the system (such as a stock price, interest rate, or important economic indicator) could be predicted with certainty one time--step ahead by agents possessing only publicly available information. Therefore, for such economic models \eqref{eq.x} is preferable to \eqref{eq.x2}. 

The equation \eqref{eq.x3} shares with \eqref{eq.x} advantageous adaptedness properties, provided 
\begin{equation} \label{eq.solve}
\text{For every $y\in\mathbb{R}$ there is $x\in\mathbb{R}$ such that $x-k(0)f(x)=y$}. 
\end{equation} 
If \eqref{eq.solve} holds, then there exists an adapted process $x$ satisfying \eqref{eq.x3}. The process is unique if there
is a unique solution to the nonlinear equation in \eqref{eq.solve} for each $y\in \mathbb{R}$. This is certainly true for all $|y|$ sufficiently large, under the sublinearity hypothesis \eqref{eq.f2}. However, we slightly prefer the formulation \eqref{eq.x} from a modelling perspective, as the summation term can represent the impact of agents on the system at time $n+1$, based on actions they make using any subset of publicly available information up to time $n$. In \eqref{eq.x3}, the value of the system at time $n$ appears both in the summation term (which we view as including information causing the future value of the output $x$) and as ``output'' itself at time $n$. While this does not violate causality in the model, it does impose on the system the additional mathematical constraint \eqref{eq.solve}
as well as its economic interpretation. In the meta--model we describe, \eqref{eq.solve} amounts to agents instantaneously solving nonlinear equations which may involve the actions of other agents at the same instant. We can, and do, avoid such problems by studying \eqref{eq.x} instead. 

\subsection{Motivation from economics}
We do not have any particular economic model in mind in formulating this equation, but merely try to capture interesting dynamical effects which seem to us to arise in economics, although we mention three situations where equations of the type \eqref{eq.x} may be germane. Our general question is: if we have a system which, although small, is relatively 
robust (in being able to handle moderate shocks), how does that system react to strong shocks or strong and persistent 
external forces? Do the shocks persist, or fade rapidly? How does the system adjust to persistent and possibly positive changes in 
the external environment? How does the memory that the system has of its own past effect the transmission of the external forces through the system over time? We are also interested in tracking quantities and time at which the solution reaches its maximum to date: such times and quantities are thought be investigators to be of psychological importance to agents.  

With these questions in mind, the structure of \eqref{eq.x} becomes more apparent. The state of the (small) system at time $n$ is $x(n)$. 
The external force or shock at time $n$ is $H(n)$. Although the form of \eqref{eq.x} does not preclude that $H(n+1)$ can be a function of 
$x$ at past states, it is tacit in our formulation that $H(n+1)$ is independent of $\{x(j):j\leq n\}$. Therefore, while $H$ influences 
$x$, $x$ does not influence $H$. In this sense, we view $x$ as modelling a ``small'' system: the external environment influences the system, but the influence of the system on the external environment is small (and modelled as being absent).

Another interpretation of $H$ is that it models the effect of ``news'' or hard--to--model external effects on the system. This is a common feature in autoregressive time series models, for instance: the system adjusts according to the previous values of the system, and is in addition, subjected to a stochastic shock which cannot be predicted with certainty based on the past states of the system. 

The Volterra term has the following interpretation: as usual for Volterra equations, we take the view that all past terms have an effect on the system, but that terms in the distant past have a vanishing impact (so $k\in \ell^1(\mathbb{N})$). The sublinearity in $f$ makes the system very robust to moderate shocks, as demonstrated by Theorem~\ref{prop.bounded} below. This is true without making any assumption on the size of $k$: in contrast, in the linear case, restrictions on $k$ would be necessary in order for solutions to remain bounded for all bounded $H$. 
Of course, if the state is an asset price or income, the system's smallness and sublinearity in $f$ is a disadvantage: it is unable to grow unboundedly by its own means (or exhibit so--called endogeneous growth). We remark in passing that if one desires endogeneous 
growth in the unperturbed system, this can be achieved by considering the difference--summation equation 
\begin{equation}  \label{eq.diffsumm}
x(n+1)-x(n)=\sum_{j=0}^n k(n-j)f(x(j)), \quad n\geq 0,
\end{equation}
If we still assume that $k\in \ell^1(\mathbb{N})$, $f$ is sublinear in the sense of \eqref{eq.f2}, and $f:(0,\infty)\to (0,\infty)$ 
and $k$ is non--negative, then all solutions of \eqref{eq.diffsumm} with positive initial condition grow to infinity at a rate 
determined by $f$. A continuous analogue of \eqref{eq.diffsumm} with these positivity assumptions is considered in~\cite{sublinear2015}.

Some existing economic models take the form of \eqref{eq.x} or are closely related to it. In the classic dynamic linear multidimensional Leontief input--output model (see e.g,~\cite{leontief1,leontief2}), $H$ is the final demand, and $x$ the output, and the Volterra term is so--called intermediate demand. The sublinearity assumption means that the (one--commodity) economy exhibits diminishing marginal returns to scale. The presence of time lags signifies that production can take many time steps to enter the final demand. 
Early examples of nonlinear input--output models include~\cite{Chander, sandberg}. 

We can think of the model in terms of an inefficient market for an  asset, where new signals about the price arrive $H(n+1)$ which drive the price, but the agents use past information about the price to determine their demand, and this also has an impact on the price. The sublinearity in this instance suggests that the traders become conservative in their net demand, relative to the price level, when the market is far from equilibrium. Our results suggest that large shocks to the price transmit quickly to the system in that case, despite the fact that the traders may use a lot of information about the past of the system. Models of this type include~\cite{anh,aik,jajd2011,jamrcs13}.  

Our model also takes inspiration from the important class of (linear) autoregressive models. The class of $\text{ARMA}(p,q)$ models 
(see e.g., \cite{BrockDavis}), for instance, have the form
\[
x(n+1)=H(n+1)+\sum_{j=n-p}^n k(n-j)x(j)
\] 
where $H$ is a stationary process which has non--trivial autocorrelation at $q\in \mathbb{Z}^+$ time lags (i.e., $\text{Cov}(H(n),H(n+q))\neq 0$ for all $n\geq 0$), and trivial autocorrelations for time lags greater than $q$ (i.e., $\text{Cov}(H(n),H(n+k))=0$ for all $n\geq 0$ and $k>q$). In this case, the equation has bounded memory of the previous $p$ values of the state. 

However so--called $\text{AR}(\infty)$ models are also considered, in which the entire history of the process is important. One motivation to do this is to introduce so--called long range dependence or long memory into the system. Classic papers finding evidence 
for slow decay in correlations in tree-ring data series, wheat
market prices, stock market and foreign exchange returns are Baillie \cite{baillie} and  Ding and Granger
\cite{zdcg:1996}). Mathematical models in economics based on $\text{AR}(\infty)$ processes have been developed.  
For instance, Kirman and Teyssi\`ere \cite{akgt:2002a, akgt:2002b} develop a time series model which arises from a market composed 
trend following and value investors which possesses long memory characteristics in the differenced log returns
of price processes associated with these models. Appleby and Krol
\cite{appkrol2} analyse the long memory properties of a linear
stochastic Volterra equation in both continuous and discrete time,
with conditions for both subexponential rates of decay and
arbitrarily slow decay rates in the autocovariance function being
characterised in terms of the decay of the kernel of the Volterra
equation. A continuous--time infinite history financial market model
is discussed in Anh et al. \cite{anh, aik}, which generalises the classic Black-Scholes model, and exhibits long memory properties. 
All these papers study equations closely related to the classic $\text{AR}(\infty)$ model: 
\[
x(n+1)=H(n+1)+\sum_{j=-\infty}^n k(n-j)x(j), \quad n\in \mathbb{Z}.
\]   
If one chooses to subsume the history of the process up to time $n=0$ in the forcing term,
and further assume (for example) that $\{x(n):n\leq 0\}$ is bounded, then the sequence 
\[ 
\tilde{H}(n+1):=H(n+1)+\sum_{j=-\infty}^{-1} k(n-j)x(j), \quad n\geq 0
\]
is well--defined and we have 
\[
x(n+1)=\tilde{H}(n+1)+\sum_{j=0}^n k(n-j)x(j), \quad n\geq 0,
\]
which is in the form of \eqref{eq.x} with $f(x)=x$. Furthermore, if the history of $x$ is bounded, $\tilde{H}(n)-H(n)$ is bounded, so the unboundedness properties of the adjusted perturbation $\tilde{H}$ and the original perturbation $H$ are the same.
 
We remark that stationarity in $H$ in these linear models does not necessarily entail stationarity in $x$: in the case of the $\text{ARMA}(p,q)$ model for example, it relies on the $\ell^1$--stability of the resolvent   
\[
r(n+1)=\sum_{j=n-p}^n k(n-j)r(j), \quad n\geq 0; \quad r(0)=1; \quad r(n)=0 \quad n<0,
\]
which is equivalent to all the zeros of the polynomial equation 
\[
z^{p+1} = \sum_{l=0}^p k(l)z^{p-l}
\]
lying in $\{z\in \mathbb{C}:|z|<1\}$. Although we have not proven it in this paper, we conjecture that stationarity in 
$H$ in \eqref{eq.x} implies asymptotic stationarity in $x$ in \eqref{eq.x}. Such a result would be in line with other results we observe here, namely that an unboundedness property $U$ in $H$ is inherited by $x$. 

\section{Main Results}
In this section we list and discuss the main results of the paper. Proofs are largely postponed to the end.
\subsection{Bounded and unbounded solutions}
Our first main result shows that if $H$ is a bounded sequence, then so is $x$, but that if $H$ is unbounded, $x$ must be also.
\begin{theorem} \label{prop.bounded}
Suppose that $f$ obeys \eqref{eq.f1} and \eqref{eq.f2}, that $k$ obeys \eqref{eq.kl1} and that $x$ is the solution of \eqref{eq.x}.
Define $H^\ast$ and $x^\ast$ as in \eqref{eq.Hstar} and \eqref{eq.xstar}. 
\begin{enumerate}
	\item[(a)] If $\lim_{n\to\infty} H^\ast(n) \in [0,\infty)$, then $\lim_{n\to\infty} x^\ast(n) \in [0,\infty)$.
	\item[(b)] If $\lim_{n\to\infty} H^\ast(n)=+\infty$, then $\lim_{n\to\infty} x^\ast(n)  = +\infty$. 
\end{enumerate}
\end{theorem}
\subsection{Growth rates in the partial maximum}
Theorem~\ref{prop.bounded} shows that solutions of \eqref{eq.x} are bounded if and only if $H$ is bounded. 
We have already noted (for growth arising from dynamic input--output models, or for unbounded shocks that would result in a time 
series model if $H$ were a stationary process) that for the applications we have mentioned, it is more natural to consider unbounded $H$. In this case $\lim_{n\to\infty} H^\ast(n)=+\infty$ and therefore $\lim_{n\to\infty} x^\ast(n)=+\infty$.

It is now a natural question to ask: if $H^\ast(n)\to\infty$ as $n\to\infty$, how rapidly will $x^\ast(n)\to\infty$ as $n\to\infty$? 
Our first result in this section shows that both maxima grow at the same rate. We also study the relationship between the times at which 
$|x|$ and $|H|$ reach their running maxima.
\begin{theorem} \label{theorem.growthmaxima}
 $f$ obeys \eqref{eq.f1} and \eqref{eq.f2} and $k$ obeys \eqref{eq.kl1}. Suppose that $x$ obeys \eqref{eq.x} and that $H$ obeys $\lim_{n\to\infty} H^\ast(n)=+\infty$. 
\begin{itemize}
\item[(i)]
$\lim_{n \to +\infty}\max_{{0}\leq{j}\leq{n}}\left|x(j)\right|=\infty$ \text{and} 
\[
\lim_{n \to +\infty} \frac{\max_{{0}\leq{j}\leq{n}}\left|x(j)\right|}{\max_{{1}\leq{j}\leq{n}}\left|H(j)\right|}=1.
\]
\item[(ii)] Let $t^H$ be defined by \eqref{eq.tstarnH}, and $t^x$  defined by \eqref{eq.tstarnx}. Then 
\begin{enumerate}
\item[(a)]
\[
\lim_{n \to +\infty}\frac{\left|x(t^H_n)\right|}{\left|H(t^H_n)\right|}=1,
\quad 
\lim_{n \to +\infty}\frac{\left|x(t^x_n)\right|}{\left|x(t^H_n)\right|}=1.
\]	
\item[(b)] 
\[
\lim_{n \to +\infty}\frac{\left|x(t^x_n)\right|}{\left|H(t^x_n)\right|}=1, \quad
\lim_{n \to +\infty}\frac{\left|H(t^x_n)\right|}{\left|H(t^H_n)\right|}=1.
\]
\end{enumerate}
\end{itemize}
\end{theorem}
In advance of proving Theorem~\ref{theorem.growthmaxima}, we now provide an interpretation of its conclusions. 

If we suppose that $H$ fluctuates such that
\[
 \max_{{1}\leq{j}\leq{n}}H(j)\to\infty \quad\text{and}\quad \min_{{1}\leq{j}\leq{n}}H(j)\to-\infty, \text{ as $n\to\infty$},
\] 
we see that part (i) implies that the order of magnitude of the large fluctuations in $x$ is precisely that of the large fluctuations in $H$.

The first limit in part (ii)(a) states that at the time to date (for large times) at which $H$ reaches its maximum, $x$ is of the same order.
Moreover, the second limit says that if one considers the epoch $\left\{0,..,n\right\}$, the largest fluctuation of $x$ is of the same order as the magnitude of $x$ at the time of the largest fluctuation of $H$. In other words, a fluctuation of the order of the biggest fluctuation in $x$ is ``caused'' at the time of the largest fluctuation in $H$, so, the largest fluctuations in $H$ transmit rapidly into the largest fluctuations in $x$.

Turning to the first limit in part (b), we see that on the epoch $\left\{0,..,n\right\}$, if the largest fluctuation in $x$ is recorded, the level of $H$ at that time is of the order of the largest fluctuation in $x$. Furthermore, the level of $H$ at that time is asymptotic to the largest fluctuation in $H$ over the epoch $\left\{0,..n\right\}$. This means that if the largest fluctuation to date in the process $x$ is observed at a specific time, then this is caused by a large fluctuation in $H$ at that time and this fluctuation in $H$ is of the order of the largest fluctuation in $H$ recorded to date.
To summarise briefly, if we observe the largest fluctuation to date in $x$, it has essentially been caused by the largest fluctuation in $H$ to date, which occurred at that time.

\subsection{Growth Rates}
Theorem~\ref{theorem.growthmaxima} shows that if $H$ is unbounded, then so is $x$, and their absolute maxima grow at the same rate. However, 
what we do not know at this point is whether growth in $H$ will produce growth in $x$, and whether fluctuations in $H$ will produce fluctuations in $x$.
In this section, we show that ``regular'' growth in $H$ (in a sense that we make precise) gives rise to regular growth in $x$, and indeed that such regular growth in $x$ is possible only if $H$ grows regularly. 

\begin{theorem} \label{theorem.growth}
Suppose $f$ obeys \eqref{eq.f1} and \eqref{eq.f2}, $k$ obeys \eqref{eq.kl1} and that $x$ is the solution of \eqref{eq.x}.
\begin{itemize} 
\item[(i)] The following statements are equivalent:
\begin{enumerate}
	\item[(a)] $H(n)$ is asymptotic to an increasing sequence and $H(n) \to \infty$ as $n\to \infty.$ 
	\item[(b)] $x(n)$ is asymptotic to an increasing sequence and $x(n)\to\infty$ as $\to\infty.$
\end{enumerate}
Moreover, both statements imply that $\lim_{n\to\infty} x(n)/H(n)=1$. 
\item[(ii)]
The following statements are equivalent:
\begin{enumerate}
	\item[(a)] $H(n)$ is asymptotic to a decreasing sequence and $H(n) \to -\infty\; as\; n\to \infty.$ 
	\item[(b)] $x(n)$ is asymptotic to a decreasing sequence and $x(n)\to-\infty\;as\;n\to\infty.$
\end{enumerate}
Moreover, both statements imply that $\lim_{n\to\infty} x(n)/H(n)=1$. 
\end{itemize}
\end{theorem}
Theorem~\ref{theorem.growth} deals with monotone growth in $H$ (and in $x$).  
If the growth is not monotone in $H$, this is also reflected in $x$. To capture non--monotone growth in $H$, with a potentially 
fluctuating component, let $(a(n))_{n\geq 1}$ be an increasing positive sequence and introduce the space of sequences $B_a$ 
\begin{equation*}
B_a=\{(H(n))_{n\geq 0} : \limsup_{n\to\infty} |H(n)|/a(n)<+\infty\}.
\end{equation*} 
It is clear that for every $y$ in $B_a$ there is a bounded sequence $\Lambda_a y$ (which is unique up to asymptotic equality) such that 
\begin{equation} \label{def.LambdaaH}
\lim_{n\to\infty} \left| \frac{y(n)}{a(n)}  - (\Lambda_a y)(n)\right|=0.
\end{equation}
We are interested in the case when $\Lambda_a H(n)$ does \emph{not} tend to a limit as $n\to\infty$: if the limit is trivial, then $a$ does not describe the rate of growth of $H$ very well and the situation is of less interest; if the limit is non--trivial, we are in the 
situation covered by Theorem~\ref{theorem.growth}.
\begin{theorem} \label{theorem.growth2}
Suppose $f$ obeys \eqref{eq.f1} and \eqref{eq.f2}, $k$ obeys \eqref{eq.kl1} and that $x$ is the solution of \eqref{eq.x}.
Let $(a(n))_{n\geq 1}$ be an increasing and positive sequence such that $a(n)\to\infty$ as $n\to\infty$. 
Then the following are equivalent
\begin{itemize}
\item[(a)] $H\in B_a$, and $\Lambda_a H$ defined by \eqref{def.LambdaaH} is asymptotically non--null; 
\item[(b)] $x\in B_a$,  and $\Lambda_a x$ defined by \eqref{def.LambdaaH} is asymptotically non--null; 
\end{itemize}
Moreover, both imply that we may take $\Lambda_a x=\Lambda_a H$.  
\end{theorem}
The interpretation of the implication (a) implies (b) of Theorem~\ref{theorem.growth2} is clear: if the external force grows at a rate $a$, modulo a non--trivial and non--constant bounded multiplicative factor $\Lambda_a H$, then the solution grows at the same rate $a$, multiplied by the factor $\Lambda_a H$. Therefore, regular growth in $H$ (with fluctuations about a trend growth rate) are reflected in $x$, and the character of the fluctuations about the trend is the same for the output $x$. Conversely, if we observe growth modified by a multiplicative fluctuation in the output $x$, this must have been caused by the same pattern of growth in the forcing term $H$.

\begin{example} \label{examp.periodicgrowth}
Let $a>0$ and suppose that $H(n)=e^{an}\pi(n)$, where $\pi$ is $N$--periodic with $\max_{i=0,\ldots, N-1} \pi(i)=\overline{\pi}>0$
and  $\min_{i=0,\ldots, N-1} \pi(i)=\underline{\pi}\in (0,\overline{\pi})$. Thus $H$ exhibits exponential growth with a periodic component, and as such is a crude model for growth with periodic booms and recessions in the world economy. The small system, whose output is influenced by $H$, is modelled by $x$. In the above notation, we can take $a(n)=e^{an}$ and $\Lambda_a H=\pi$. Then by 
Theorem~\ref{theorem.growth2}
\[
\lim_{n\to\infty}\left|\frac{x(n)}{e^{an}}-\pi(n)\right|=0,
\]    
so we see that $x$ inherits the main properties of the growth path of $H$: in economic terms, the booms and recessions in the outside system propagate rapidly into the small system. 
\end{example}

\subsection{Signed fluctuations and their magnitudes}
We have just seen that Theorem~\ref{theorem.growthmaxima}, while useful, does not distinguish between growth or fluctuations in solutions of \eqref{eq.x}. Theorem \ref{theorem.growth} demonstrates that regular growth in $H$ gives rise to regular growth in $x$ at the same rate as $H$. 
The question at hand now is to refine, in a similar manner, Theorem~\ref{theorem.growthmaxima}, in order to capture the large fluctuations in solutions of \eqref{eq.x}. It is reasonable to suppose that such fluctuations in $x$ must result from large fluctuations in $H$, and in parallel 
with Theorem~\ref{theorem.growth}, it is also reasonable to try to connect the sizes of the large fluctuations in $x$ to those in $H$. 

We have used the term fluctuation loosely above, but now we want to try to capture it mathematically. We are assuming that $H^\ast(n)\to \infty$ 
as $n\to\infty$, but in order to describe a fluctuation in $H$, we do not want to have $H(n)\to\infty$ or $H(n)\to-\infty$ as $n\to\infty$, or more generally, we do not want the limit of $H$ to exist. Roughly speaking, we could have two types of fluctuation in $H$: the first type, which we emphasise here, is that $H$ fluctuates without bound to plus and minus infinity. The second is that $H$ has an infinite limsup but finite liminf (or negative infinite liminf and finite limsup). 

Considering the first situation a little more, we should distinguish between the sizes of large positive and large negative fluctuations in $H$. To this end, we introduce the monotone sequences  
\begin{gather} \label{def.Hastpm}
	H^*_+(n):=\max_{{1}\leq{j}\leq{n}}H(j), \qquad
	H^*_-(n):=-\min_{{1}\leq{j}\leq{n}}H(j)=\max_{{1}\leq{j}\leq{n}}(-H(j)). 
\end{gather}
We see that $H^\ast_+$ records the magnitude of the large positive fluctuations, while $H^\ast_-$ records the magnitude of the large negative  fluctuations. Clearly the overall maximum of these magnitudes is just $H^\ast$, or $H^*(n)=\max(H^*_+(n),H^*_-(n))$. 

We expect that fluctuations in $H$ will cause fluctuations in $x$, so we make make the corresponding definitions for $x$ as well.
These are
\begin{gather}  \label{def.xastpm}
x^*_+(n):=\max_{{0}\leq{j}\leq{n}}x(j), \qquad
x^*_-(n):=-\min_{{0}\leq{j}\leq{n}}x(j)=\max_{{0}\leq{j}\leq{n}}(-x(j)),
\end{gather}
and $x^*(n)=\max\left(x^*_+(n),x^*_-(n)\right)$.

We are now in a position to state and prove our main result, Theorem~\ref{theorem.signedfluc} below. It is useful to assume that there is $\lambda\in [0,\infty]$ 
such that 
\begin{equation} \label{eq.Hlambda} 
\lambda:=\lim_{n\to\infty}\frac{H^*_-(n)}{H^*_+(n)}. 
\end{equation}
The existence of this limit helps us to decide whether the large negative or large positive fluctuations dominate.

If the large positive fluctuations 
in $H$ dominate asymptotically the large negative fluctuations (in the sense that $\lambda\in [0,1)$ in \eqref{eq.Hlambda}) then $x$ experiences a large positive fluctuation of the same order as the large positive fluctuation in $H$, and this also captures growth rate of the partial maximum of $|x|$; in other words, if $x$ experiences a large negative fluctuation, it will be dominated by the large positive fluctuation. This is the subject of  part (i) in Theorem~\ref{theorem.signedfluc}.

Symmetrically, if the large negative fluctuations 
in $H$ dominate asymptotically the large positive fluctuations (in the sense that $\lambda\in (1,\infty]$ in \eqref{eq.Hlambda}), then $x$ experiences a large negative fluctuation of the same order as the large negative fluctuation in $H$, and this also captures growth rate of the partial maximum of $|x|$; in other words, if $x$ experiences a large positive fluctuation, it will be dominated by the large negative fluctuation. 
This is the subject of  part (ii) in Theorem~\ref{theorem.signedfluc}.

Finally, if the growth rates of the the large positive and large negative fluctuations in $H$ are the same (in the sense that $\lambda=1$ in \eqref{eq.Hlambda}), then $x$ experiences both large positive and large negative fluctuations, the growth rate of both are the same, and moreover equal to the growth rates of the fluctuations in $H$. This is the subject of  part (iii) in Theorem~\ref{theorem.signedfluc}.

\begin{theorem} \label{theorem.signedfluc}
Suppose $f$ obeys \eqref{eq.f1} and \eqref{eq.f2}, $k$ obeys \eqref{eq.kl1} and that $x$ is the solution of \eqref{eq.x}.
Suppose $H^*(n)\to\infty$ as $n\to\infty$, and that $H$ and $\lambda$ obey \eqref{eq.Hlambda}, and that $H^\ast_{\pm}$ and $x^\ast_{\pm}$ are defined by \eqref{def.Hastpm} and \eqref{def.xastpm}. 
\begin{itemize}
\item[(i)] If $\lambda\in [0,1)$ then
\begin{equation*} \lim_{n\to\infty} x^\ast_+(n)=+\infty  \end{equation*}
and
\begin{equation*} \lim_{n\to\infty} \frac{x^\ast_+(n)}{H^\ast_+(n)}
=\lim_{n\to\infty} \frac{x^\ast(n)}{H^\ast_+(n)}=1.\end{equation*}
\item[(ii)] If $\lambda\in (1,\infty]$ then
\begin{equation*} \lim_{n\to\infty} x^\ast_-(n)=+\infty
\end{equation*}
 and 
\begin{equation*} 
\lim_{n\to\infty} \frac{x^\ast_-(n)}{H^\ast_-(n)}
=\lim_{n\to\infty} \frac{x^\ast(n)}{H^\ast_-(n)}=1.
\end{equation*}
\item[(iii)]  If $\lambda= 1$ then 
\begin{equation*}
\lim_{n\to\infty} x^\ast_+(n)=+\infty, \quad \lim_{n\to\infty} x^\ast_-(n)=+\infty,
\end{equation*}
and
\begin{equation*} 
\lim_{n\to\infty} \frac{x^\ast_+(n)}{H^*_+(n)}=\lim_{n\to\infty} \frac{x^\ast_-(n)}{H^*_-(n)}=1. 
\end{equation*}
\end{itemize}
\end{theorem}
We note an asymmetry here in parts (i) and (ii) between assumptions on $H$ and conclusions concerning $x$. If $\lambda\in (0,\infty)$, we have both $H^\ast_+(n)$ and $H^\ast_-(n)\to\infty$ as $n\to\infty$. However, part (i) only yields $x^\ast_+(n)\to\infty$ in the case when $\lambda\in (0,1)$, while part (ii) yields only $x^\ast_-(n)\to\infty$ in the case when $\lambda\in (1,\infty)$. In other words, despite the fact that $H$ experiences negative fluctuations in part (i), we do not say anything about corresponding large negative fluctuations in $x$, and in part (ii), large positive fluctuations in $H$ do not give us any conclusions about the presence of large positive fluctuations in $x$. Further analysis shows that this limitation can be overcome: the results are summarised in the next theorem. 

Roughly speaking, if the large positive and large negative fluctuations are of the same order of magnitude, $x$ experiences both large positive and large negative fluctuations, the large positive fluctuations of $x$ grow at exactly the rate of the 
positive fluctuations of $H$, and the negative fluctuations grow at exactly the same rate as those of $H$. In the case that 
the positive fluctuations of $H$ dominate the negative fluctuations, 
the positive fluctuations of $x$ dominate its negative fluctuations.  
Finally, if the negative fluctuations of $H$ dominate its positive fluctuations, the negative fluctuations of $x$ dominate 
its positive fluctuations.  
\begin{theorem} \label{theorem.signedfluc2}
Suppose $f$ obeys \eqref{eq.f1} and \eqref{eq.f2}, $k$ obeys \eqref{eq.kl1} and that $x$ is the solution of \eqref{eq.x}.
Suppose $H^*(n)\to\infty$ as $n\to\infty$, and that $H$ and $\lambda$ obey \eqref{eq.Hlambda}, and that $H^\ast_{\pm}$ and $x^\ast_{\pm}$ are defined by \eqref{def.Hastpm} and \eqref{def.xastpm}.
\begin{itemize}
\item[(i)] If $\lambda\in (0,\infty)$,  then
\begin{equation*} \lim_{n\to\infty} x^\ast_+(n)=+\infty, 
\quad \lim_{n\to\infty} x^\ast_-(n)=+\infty  
\end{equation*}
and
\begin{equation*} 
\lim_{n\to\infty} \frac{x^\ast_+(n)}{H^\ast_+(n)}=1, 
\quad 
\lim_{n\to\infty} \frac{x^\ast_-(n)}{H^\ast_-(n)}=1.
\end{equation*}
\item[(ii)] If $\lambda=0$,  then
\begin{equation*} 
\lim_{n\to\infty} x^\ast_+(n)=+\infty,
\end{equation*}
 and
\begin{equation*} 
\lim_{n\to\infty} \frac{x^\ast_+(n)}{H^\ast_+(n)}=1, 
\quad 
\lim_{n\to\infty} \frac{x^\ast_-(n)}{H^\ast_+(n)}=0.
\end{equation*}
\item[(iii)]  If $\lambda= \infty$, then 
\begin{equation*}
\lim_{n\to\infty} x^\ast_-(n)=+\infty,
\end{equation*}
and
\begin{equation*} 
\lim_{n\to\infty} \frac{x^\ast_+(n)}{H^\ast_-(n)}=0, 
\quad 
\lim_{n\to\infty} \frac{x^\ast_-(n)}{H^\ast_-(n)}=1.
\end{equation*}
\end{itemize}
\end{theorem}
We note that part (ii) does not allow us to conclude that $\liminf_{n\to\infty} x(n)=-\infty$ under the condition that 
$\liminf_{n\to\infty} H(n)=-\infty$. We now show that by strengthening hypothesis on $f$ and $H$ that it is possible to conclude more. 
However, we also show that $\liminf_{n\to\infty} H(n)=-\infty$ does not necessarily imply that $\liminf_{n\to\infty} x(n)=-\infty$. 

We now make our additional assumptions on $f$: we assume that 
\begin{equation} \label{eq.fphi}
\text{There is $\phi\in C([0,\infty),[0,\infty))$ which is increasing and obeys } 
\lim_{|x|\to\infty} \frac{|f(x)|}{\phi(|x|)}=1.
\end{equation}
We also ask that 
\begin{equation}  \label{eq.phixdec}
\text{$x\mapsto \phi(x)/x$ is decreasing on $[x_0,\infty)$.}
\end{equation}
We invoke symmetry in $f$ to see what the impact is when there is asymmetry in the growth of $H^\ast_{\pm}(n)$ as $n\to\infty$.  
We deal only with the case in which the positive fluctuations dominate the negative ones: an analogous result can be obtained by the same means if the negative fluctuations are dominant (i.e., when $\lambda=+\infty$).
\begin{theorem} \label{theorem.signedfluct}
Suppose $f$ obeys \eqref{eq.f1} and \eqref{eq.f2}, $k$ obeys \eqref{eq.kl1} and that $x$ is the solution of \eqref{eq.x}.
Suppose further that there is $\phi$ such that $f$ and $\phi$ obey \eqref{eq.fphi} and \eqref{eq.phixdec}. 
Suppose $H^*(n)\to\infty$ as $n\to\infty$, and that $H$ and $\lambda$ obey \eqref{eq.Hlambda}, and that $H^\ast_{\pm}$ and $x^\ast_{\pm}$ are defined by \eqref{def.Hastpm} and \eqref{def.xastpm}.
\begin{itemize}
\item[(i)] If $\lambda\in (0,\infty)$,  then
\begin{equation} \label{eq.xpxminfty}
\lim_{n\to\infty} x^\ast_+(n)=+\infty, 
\quad \lim_{n\to\infty} x^\ast_-(n)=+\infty  
\end{equation}
and
\begin{equation} \label{eq.xpxminftyHpHm}
\lim_{n\to\infty} \frac{x^\ast_+(n)}{H^\ast_+(n)}=1, 
\quad 
\lim_{n\to\infty} \frac{x^\ast_-(n)}{H^\ast_-(n)}=1.
\end{equation}
\item[(ii)] Let $\lambda=0$, and suppose that
\begin{equation} \label{eq.lambda2}
\lim_{n\to\infty} \frac{H_-^\ast(n)}{f(H_+^\ast(n))}=:\lambda_2\in [0,\infty]. 
\end{equation}
If $\lambda_2=+\infty$, then \eqref{eq.xpxminfty} and \eqref{eq.xpxminftyHpHm} hold.
\item[(iii)] If $\lambda=0$ and $\lambda_2\in (0,\infty)$, then $\lim_{n\to\infty} x^\ast_+(n)=+\infty$ and
\begin{equation*} 
\lim_{n\to\infty} \frac{x^\ast_+(n)}{H^\ast_+(n)}=1, 
\quad 
\limsup_{n\to\infty} \frac{x^\ast_-(n)}{H^\ast_-(n)}\leq 1+\frac{1}{\lambda_2}\sum_{j=0}^\infty |k(j)|.
\end{equation*}
If, in addition, $\lambda_2>\sum_{j=0}^\infty |k(j)|$, then \eqref{eq.xpxminfty} holds and 
\[
 1-\frac{1}{\lambda_2}\sum_{j=0}^\infty |k(j)|\leq 
\liminf_{n\to\infty} \frac{x^\ast_-(n)}{H^\ast_-(n)}\leq 
\limsup_{n\to\infty} \frac{x^\ast_-(n)}{H^\ast_-(n)}\leq 1+\frac{1}{\lambda_2}\sum_{j=0}^\infty |k(j)|.
\] 
\item[(iv)] If $\lambda=0$ and $\lambda_2=0$, then $\lim_{n\to\infty} x^\ast_+(n)=+\infty$ and
\[
\lim_{n\to\infty} \frac{x^\ast_+(n)}{H^\ast_+(n)}=1, 
\quad 
\limsup_{n\to\infty} \frac{x^\ast_-(n)}{f(H^\ast_+(n))}\leq \sum_{j=0}^\infty |k(j)|.
\]
\end{itemize}
\end{theorem}
We see (roughly) that if the negative fluctuations are not too small relative to the positive fluctuations, then it is still the case 
that the size of the negative fluctuations of $H$ determine the size of the negative fluctuations of $x$. However, if the negative fluctuations become too small relative to the positive fluctuations, it need not be the case that $x^\ast_-(n)$ is of the same order of magnitude as $H^\ast_-(n)$, as the following example shows. Indeed, it can even happen that the negative fluctuations of $x$ are bounded, even though those of $H$ are unbounded. We also see that the estimates in part (ii) and (iii) are better than the crude estimate $x^\ast_-(n)=o(H^\ast_+(n))$ as $n\to\infty$ that is supplied by part (ii) of Theorem~\ref{theorem.signedfluc2}; notice also how part (ii) of Theorem~\ref{theorem.signedfluct} can be thought of as a limiting case of part (iii) when $\lambda_2\to \infty$.
\begin{example}
Let $\mu_+>\mu_-$ where $\mu_+>0$ and $\mu_-\geq 0$. Let $\alpha\in (0,1)$. Let $n\geq 0$. 
Define $y(n)=(n+1)^{\mu_+}$ for $n$ an even integer and $y(n)=-(n+1)^{\mu_-}$ for $n$ an odd integer. 
Let $(k(n))_{n\geq 0}$ be a non--negative summable sequence and suppose $f(x)=\sgn(x)|x|^\alpha$ for $x\in\mathbb{R}$.
Define for $n\geq 0$
\[
H(n+1):=y(n+1)-\sum_{j=0}^n k(n-j) f(y(j)).
\]
Then $x(n)=y(n)$ for $n\geq 0$ is a solution of \eqref{eq.x}. $f$ obeys all the hypothesis of Theorem~\ref{theorem.signedfluct}.

If $\mu_+\alpha<\mu_-$, one can show that $H(n)\sim -n^{\mu_-}$ as $n\to\infty$ through the odd integers, and that 
$H(n)\sim n^{\mu_+}$ as $n\to\infty$ through the even integers. Then we see that $H^\ast_-(n)\sim n^{\mu_-}$ as $n\to\infty$ 
and $H^\ast_+(n)\sim n^{\mu_+}$ as $n\to\infty$. Thus $H^\ast_-(n)/f(H^\ast_+(n))\to 0$ as $n\to\infty$, and hence 
we have from part (ii) of Theorem~\ref{theorem.signedfluct} that $x^\ast_+(n)\sim n^{\mu_+}$ and $x^\ast_-(n)\sim n^{\mu_-}$ 
as $n\to\infty$. Of course, we can conclude this independently of Theorem~\ref{theorem.signedfluct} by simply observing that 
$x(n)=(n+1)^{\mu_+}$ for $n$ even and $x(n)=-(n+1)^{\mu_-}$ for $n$ odd.  

If $\mu_-<\alpha\mu_+$, we can show that $H(n)\sim -\sum_{j=0}^\infty k(2j) \cdot n^{\mu_+\alpha}$ as $n\to\infty$ through the odd integers, and $H(n)\sim n^{\mu_+}$ as $n\to\infty$ through the even integers. Hence $H^\ast_-(n)\sim \sum_{j=0}^\infty k(2j) \cdot n^{\mu_+\alpha}$ as $n\to\infty$ and $H^\ast_+(n)\sim n^{\mu_+}$ as $n\to\infty$. We know, by construction, that $x^\ast_+(n)\sim n^{\mu_+}$ 
and $x^\ast_-(n)\sim (1+n)^{\mu_-}$ as $n\to\infty$. We also have 
\[
\lim_{n\to\infty} \frac{H^\ast_-(n)}{f(H^\ast_+(n))} =\sum_{j=0}^\infty k(2j). 
\] 
Thus $\lambda_2=\sum_{j=0}^\infty k(2j) <\sum_{j=0}^\infty k(j)$ (provided $k(n)\neq 0$ for some odd integer $n$). Therefore, we can apply part (iii) of Theorem~\ref{theorem.signedfluct} to see that $x^\ast_-(n)=\mathcal{O}(n^{\mu_+\alpha})$ as $n\to\infty$. 

Now some remarks are in order. First, we see that part (iii) of Theorem~\ref{theorem.signedfluct} is not necessarily sharp, because we know that $x^\ast_-(n)\sim (1+n)^{\mu_-}$ as $n\to\infty$, while the upper bound furnished by the theorem grows strictly faster. Second, the rather precise estimates on the asymptotic behaviour of $H^\ast_{\pm}(n)$ as $n\to\infty$ that we possess (i.e., $H^\ast_-(n)\sim \sum_{j=0}^\infty k(2j) \cdot n^{\mu_+\alpha}$ as $n\to\infty$ and $H^\ast_+(n)\sim n^{\mu_+}$) are not sufficient to be able to predict the exact rate $x^\ast_-(n)\sim (1+n)^{\mu_-}$ as $n\to\infty$: the parameter $\mu_-$ does not appear in $k$, $f$ nor in the leading order asymptotic behaviour in $H^\ast_{\pm}$. Third, since $\mu_-=0$ is an admissible parameter value, it is the case that $H^\ast_-(n)\to\infty$ as $n\to\infty$ does \emph{not} necessarily imply that $x^\ast_-(n)\to\infty$ as $n\to\infty$. Fourth, it is the case that  
\[
\frac{x^\ast_-(n)}{H^\ast_-(n)}\sim \frac{n^{\mu_-}}{\sum_{j=0}^\infty k(2j) \cdot n^{\mu_+\alpha}} \to 0 \text{ as $n\to\infty$}, 
\]
so in general, we see that it is not necessarily the case that $x^\ast_-(n)$ is of the same order of magnitude as $H^\ast_-(n)$ 
as $n\to\infty$. Fifth, even though the upper bound in part (iii) is  not sharp, the Theorem does seem to identify, in 
this example at least, that there is a change in the asymptotic behaviour of $x^\ast_-$ when $\mu_-$ moves from being greater than $\alpha \mu_+$ to being less than $\alpha\mu_+$.
\end{example}


\subsection{Bounds on the fluctuations in terms of an auxiliary sequence} 
In applications, especially when $H$ is a stochastic process, it may be possible to prove by independent methods that there are increasing deterministic sequences which give precise bounds on the fluctuations of $H$. In the important case where $H$ is a sequence of independent and identically distributed random variables, it is possible to prove, by means of the Borel--Cantelli lemmas, that 
there exist sequences $a_+$ and $a_-$, which have very similar (but non--identical) asymptotic behaviour such that
\[
\limsup_{n\to\infty} \frac{|H(n)|}{a_+(n)}=0, \quad \limsup_{n\to\infty} \frac{|H(n)|}{a_-(n)}=+\infty, \quad \text{a.s.}
\]
An example of a case where this holds is when each $H(n)$ has the power law density $g(x)\sim c_\alpha|x|^{-\alpha}$ and $\alpha>1$. 
In this case one can take for instance $a_+(n)=n^{1/(\alpha-1)+\epsilon}$ and $a_-(n)=n^{1/(\alpha-1)-\epsilon}$ for any $\epsilon>0$ sufficiently small. In some cases one can even show that a single sequence determines the asymptotic behaviour, so it is possible to show that 
\[
\limsup_{n\to\infty} \frac{|H(n)|}{a(n)}=1, \quad \text{a.s.}
\]
An example for which this is true is a zero mean Gaussian white noise sequence, in which $a(n)=\sigma \sqrt{2\log n}$, where $\sigma^2$ 
is the variance of the white noise process. We give details of the calculations in the next subsection. 

These examples show that the auxiliary sequence $a$ may exactly estimate the fluctuations of $H$, or systematically over-- or 
underestimate it. Therefore, it makes sense to formulate a result in which 
$\limsup_{n\to\infty} |H(n)|/a(n)$ 
can be zero, finite but non--zero, or infinite, and attempt therefrom to determine the asymptotic behaviour of $|x|$. The following result shows, once again, the close coupling of the asymptotic behaviour of $H$ and $x$.
\begin{theorem} \label{theorem.limsup}
Suppose $f$ obeys \eqref{eq.f1} and \eqref{eq.f2}, and $k$ obeys \eqref{eq.kl1}. Let $x$ be the solution $x$ of \eqref{eq.x}. 
Suppose that $(a(n))_{n\geq 1}$ is an increasing sequence such that with $a(n)\to\infty$ as $n\to\infty$. Then the following are equivalent:
\begin{itemize}
\item[(a)] There exists $\rho\in [0,\infty]$ such that 
$\limsup_{n\to\infty} |H(n)|/a(n)=\rho$; 
\item[(b)]  There exists $\rho\in [0,\infty]$ such that 
$\limsup_{n\to\infty} H^\ast(n)/a(n)=\rho$; 
\item[(c)] There exists $\rho\in [0,\infty]$ such that 
$\limsup_{n\to\infty} |x(n)|/a(n)=\rho$; 
\item[(d)] There exists $\rho\in [0,\infty]$ such that 
$\limsup_{n\to\infty} x^\ast(n)/a(n)=\rho$. 
\end{itemize}
\end{theorem}
In the case when $\rho\in(0,\infty)$, large fluctuations of both $H$ and $x$ are described by the increasing sequence $\rho a$. 
If however, a sequence $a$ does not exist (or cannot readily be found) for which this holds, a very easy corollary of 
Theorem~\ref{theorem.limsup} gives upper and lower bounds on the fluctuations of $x$ in terms of those of $H$.
\begin{theorem} \label{theorem.limsup2}
Suppose $f$ obeys \eqref{eq.f1} and \eqref{eq.f2}, and $k$ obeys \eqref{eq.kl1}. 
Suppose also that there exist increasing sequences $(a_-(n))_{n\geq 1}$ and $(a_+(n))_{n\geq 1}$ with $a_{\pm}(n)\to\infty$ as $n\to\infty$ such that
\begin{equation*} 
\limsup_{n\to\infty}\frac{\left|H(n)\right|}{a_+(n)}=0, \quad 
\limsup_{n\to\infty}\frac{\left|H(n)\right|}{a_-(n)}=+\infty.  
\end{equation*}
Then the solution $x$ of \eqref{eq.x} obeys
\begin{equation*} 
\limsup_{n\to\infty}\frac{\left|x(n)\right|}{a_+(n)}=0, \quad 
\limsup_{n\to\infty}\frac{\left|x(n)\right|}{a_-(n)}=+\infty.  
\end{equation*}
\end{theorem}
\begin{proof}
Take $a(n)=a_+(n)$ and note that $\rho=0$ in Theorem~\ref{theorem.limsup}. Applying Theorem~\ref{theorem.limsup}
gives the first limit in the conclusion of the result. The second limit is obtained by taking $a(n)=a_-(n)$, in which case 
$\rho=+\infty$, and Theorem~\ref{theorem.limsup} can be applied again.  
\end{proof}

It is equally reasonable to formulate results for the size of the positive and negative fluctuations in terms of auxiliary sequences. This result parallels Theorem~\ref{theorem.signedfluc2}. Rather than being comprehensive at the expense of repetition, we have considered the case when the positive fluctuations dominate the negative ones. Other results in this direction can be readily formulated and proven as desired using the same methods of proof: this result can be thought of as being representative. Applications of this result to Gaussian and heavy--tailed distributions are given in the next subsection.
\begin{theorem} \label{theorem.limsupsignedfluc2}
Suppose $f$ obeys \eqref{eq.f1} and \eqref{eq.f2}, $k$ obeys \eqref{eq.kl1} and that $x$ is the solution of \eqref{eq.x}.
Suppose also that there exist increasing sequences $(a_-(n))_{n\geq 1}$ and $(a_+(n))_{n\geq 1}$ with $a_{\pm}(n)\to\infty$ as $n\to\infty$ such that
\begin{equation*} 
\limsup_{n\to\infty}\frac{H(n)}{a_+(n)}=:\rho_+\in (0,\infty], \quad 
\liminf_{n\to\infty}\frac{H(n)}{a_-(n)}=:-\rho_-\in (-\infty,0], 
\end{equation*}
and 
\[
\lim_{n\to\infty} \frac{a_-(n)}{a_+(n)}=:\lambda\in [0,\infty).
\]
\begin{itemize}
\item[(i)] If $\lambda\in (0,\infty)$,  then
\begin{equation*} \limsup_{n\to\infty} x(n)=+\infty, 
\quad \liminf_{n\to\infty} x(n)=-\infty  
\end{equation*}
and
\begin{equation*} 
\limsup_{n\to\infty} \frac{x(n)}{a_+(n)}=\rho_+, 
\quad 
\liminf_{n\to\infty} \frac{x(n)}{a_-(n)}=-\rho_-.
\end{equation*}
\item[(ii)] If $\lambda=0$,  then
\begin{equation*} 
\limsup_{n\to\infty} x(n)=+\infty,
\end{equation*}
 and
\begin{equation*} 
\limsup_{n\to\infty} \frac{x(n)}{a_+(n)}=\rho_+, 
\quad 
\liminf_{n\to\infty} \frac{x(n)}{a_+(n)}=0. 
\end{equation*}
\end{itemize}
\end{theorem} 

\subsection{Applications to stochastic processes} 
Let $H(n)$ be a sequence of independent and identically distributed random 
variables each with distribution function $F$. For simplicity suppose that the distribution is continuous and supported on all of $\mathbb{R}$ (so that the random variables are unbounded and can take arbitrarily large positive and negative values). What follows is all well--known, but we record our conclusions to assist stating applying our results, which we do momentarily. 

Since each $H$ has distribution function $F$ we have  
\[
\mathbb{P}[|H(n)|>Ka(n)]=1-F(Ka(n))+F(-Ka(n)).
\]
Define 
\[
S(a,K)=\sum_{n=0}^\infty \{1-F(Ka(n))+F(-Ka(n))  \}.
\]
Since the events $\{|H(n)|>Ka(n)\}$ are independent, we have that from the Borel--Cantelli Lemma that 
\[
\mathbb{P}[|H(n)|>Ka(n) \text{ i.o.}]
=\left\{\begin{array}{cc}
0, & \text{if $S(a,K)<+\infty$}, \\
1, & \text{if $S(a,K)=+\infty$}. 
\end{array}
\right.
\]
Therefore, for all $K>0$ such that $S(a,K)<+\infty$ we have that there is an a.s. event $\Omega_K^+$ such that
\[
\limsup_{n\to\infty} \frac{|H(n)|}{a(n)} \leq K, \quad\text{on $\Omega_K^+$}.
\]
On the other hand, for all $K>0$ such that  $S(a,K)=+\infty$ we have that there is an a.s. event $\Omega_K^-$ such that
\[
\limsup_{n\to\infty} \frac{|H(n)|}{a(n)} \geq K, \quad\text{on $\Omega_K^-$}.
\]
It can be seen therefore that it may be possible for a well--chosen sequence $a$ and number $K$ sequence $Ka(n)$ for which $S(a,K)$ is either finite or infinite. This will then generate upper and lower bounds on the growth of $|H(n)|$, and thereby, by then applying Theorem~\ref{theorem.limsup}, allow conclusions about the growth of the fluctuations of $x$ to be deduced.

In the first example, we are able to find a sequence $a$ for which $\Lambda_a|H|\in (0,\infty)$. 
\begin{example} \label{examp.normal}
Suppose that $H(n)$ is a sequence of independent normal random variables with mean zero and variance $\sigma^2>0$. 
Take $a(n)=\sqrt{2 \log n}$. Then it is well--known for every $\epsilon\in (0,\sigma)$ that we have  
\[
S(a,\sigma+\epsilon)<+\infty, \quad S(a,\sigma-\epsilon)=+\infty.
\]
Therefore, there are a.s. events $\Omega_\epsilon^{\pm}$ such that 
\[
\limsup_{n\to\infty} \frac{|H(n)|}{\sqrt{2\log n}} \geq \sigma-\epsilon, \quad \text{a.s. on $\Omega_\epsilon^-$} 
\]
and
\[
\limsup_{n\to\infty} \frac{|H(n)|}{\sqrt{2\log n}} \leq \sigma+\epsilon, \quad \text{a.s. on $\Omega_\epsilon^+$} 
\]
Now consider 
\[
\Omega^\ast=\cap_{\epsilon\in \mathbb{Q}\cap(0,\sigma)} \Omega_\epsilon^+ \cap \cap_{\epsilon\in \mathbb{Q}\cap\cap(0,\sigma)} \Omega_\epsilon^-.
\]
Then $\Omega^\ast$ is an almost sure event and we have 
\[
\limsup_{n\to\infty} \frac{|H(n)|}{\sqrt{2\log n}} = \sigma, \quad \text{on $\Omega^\ast$}. 
\]
Hence we can apply Theorem~\ref{theorem.limsup} to \eqref{eq.x} with $a(n)=\sqrt{2\log n}$ to get that 
\[
\limsup_{n\to\infty} \frac{|x(n)|}{\sqrt{2\log n}} =\sigma, \quad \text{a.s.}
\]
A similar argument applies to signed fluctuations as well. We can use the Borel--Cantelli lemmas to prove that 
\[
\limsup_{n\to\infty} \frac{H(n)}{\sqrt{2\log n}}=\sigma, \quad \liminf_{n\to\infty} \frac{H(n)}{\sqrt{2\log n}}=-\sigma, \quad \text{a.s.}
\]
Therefore, by Theorem~\ref{theorem.limsupsignedfluc2} we get 
\[
\limsup_{n\to\infty} \frac{x(n)}{\sqrt{2\log n}}=\sigma, \quad \liminf_{n\to\infty} \frac{x(n)}{\sqrt{2\log n}}=-\sigma, \quad \text{a.s.}
\]
\end{example}

Next we consider the case of a symmetric heavy tailed distribution with power law decay in the tails. In this case, we find sequences 
$a_+$ and $a_-$ such that $a_-=o(a_+)$ and 
\[
\limsup_{n\to\infty} \frac{|H(n)|}{a_+(n)}=0, \quad \limsup_{n\to\infty} \frac{|H(n)|}{a_-(n)}=+\infty, \quad \text{a.s.}
\]
Even though $a_+$ dominates $a_-$ asymptotically, $a_+$ and $a_-$ will have very similar asymptotic behaviour. It follows from Theorem~\ref{theorem.limsup2} that 
\[
\limsup_{n\to\infty} \frac{|x(n)|}{a_+(n)}=0, \quad \limsup_{n\to\infty} \frac{|x(n)|}{a_-(n)}=+\infty, \quad \text{a.s.}
\]

\begin{example} \label{examp.heavytail}
Suppose that $H(n)$ are independently and identically distributed random variables such that there is $\alpha>0$ and finite $c_1,c_2>0$ 
for which
\[
F(x)\sim c_1 |x|^{-\alpha}, \quad x\to -\infty, \quad 1-F(x)\sim c_2 x^{-\alpha}, \quad x\to +\infty.  
\]
Suppose that $a_+$ and $a_-$ are sequences such that 
\begin{equation} \label{eq.apaminpower}
\sum_{n=0}^\infty a_+(n)^{-\alpha} <+\infty, \quad \sum_{n=0}^\infty a_-(n)^{-\alpha} =+\infty.
\end{equation}
Then we see that $S(K,a_+)<+\infty$ for all $K>0$ while $S(K,a_-)=+\infty$ for all $K> 0$. Therefore we have for all $K>0$
\[
\limsup_{n\to\infty} \frac{|H(n)|}{a_+(n)} \leq K, \quad\text{on $\Omega_K^+$}.
\]
Consider the event $\Omega^+=\cap_{K\in \mathbb{Q}^+} \Omega_K^+$. Then $\Omega^+$ is an almost sure event and we have 
\[
\limsup_{n\to\infty} \frac{|H(n)|}{a_+(n)}=0,  \quad\text{on $\Omega^+$}.
\]
On the other hand, for all $K>0$ we have that there is an a.s. event $\Omega_K^-$ such that
\[
\limsup_{n\to\infty} \frac{|H(n)|}{a_-(n)} \geq K, \quad\text{on $\Omega_K^-$}.
\]
Consider the event $\Omega^-=\cap_{K\in \mathbb{Z}^+} \Omega_K^+$. Then $\Omega^-$ is an almost sure event and we have 
\[
\limsup_{n\to\infty} \frac{|H(n)|}{a_-(n)}=+\infty,  \quad\text{on $\Omega^-$}.
\]
Finally, let $\Omega^\ast=\Omega^+\cap\Omega^-$. It is an almost sure event and we have that 
\[
\limsup_{n\to\infty} \frac{|H(n)|}{a_+(n)}=0,  \quad  \limsup_{n\to\infty} \frac{|H(n)|}{a_-(n)}=\infty \text{ on $\Omega^\ast$}.
\] 
Applying Theorem~\ref{theorem.limsup2} we therefore see that \eqref{eq.apaminpower} implies 
\[
\limsup_{n\to\infty} \frac{|x(n)|}{a_+(n)}=0, \quad \limsup_{n\to\infty} \frac{|x(n)|}{a_-(n)}=+\infty, \text{ on $\Omega^\ast$}.
\]
By similar arguments we can obtain bounds on the signed fluctuations as well. In fact   \eqref{eq.apaminpower} implies 
\begin{gather*}
\limsup_{n\to\infty} \frac{x(n)}{a_+(n)}=0, \quad \limsup_{n\to\infty} \frac{x(n)}{a_-(n)}=+\infty, \quad \text{a.s.}\\
\liminf_{n\to\infty} \frac{x(n)}{a_+(n)}=0, \quad \liminf_{n\to\infty} \frac{x(n)}{a_-(n)}=-\infty, \quad \text{a.s.}
\end{gather*}

To show we can get $a_+$ and $a_-$ close, notice that for every $\epsilon>0$ sufficiently small we can take 
$a_{\pm}(n)$ to be $a_{\pm\epsilon}(n)=n^{1/\alpha\pm\epsilon}$. 

It is now standard to get limits independent of the small parameter $\epsilon$, and we show now how this can be done. First, from the existence of the sequences $a_{\pm\epsilon}$ we may conclude from that 
there are a.s. events $\Omega_\epsilon^-$ and $\Omega_\epsilon^+$ such that
\[
\limsup_{n\to\infty} \frac{|x(n)|}{n^{1/\alpha -\epsilon}}=+\infty, \text{ on $\Omega_\epsilon^-$}
\]
and 
\[
\limsup_{n\to\infty} \frac{|x(n)|}{n^{1/\alpha +\epsilon}}=0, \text{ on $\Omega_\epsilon^+$}.
\]
Now we seek $\epsilon$--independent limits. We conclude from the first limit that 
\[
\limsup_{n\to\infty} \frac{\log |x(n)|}{\log n} \geq \frac{1}{\alpha}-\epsilon, \text{ on $\Omega_\epsilon^-$}
\]
and from the second that
\[
\limsup_{n\to\infty} \frac{\log |x(n)|}{\log n} \leq \frac{1}{\alpha}+\epsilon, \text{ on $\Omega_\epsilon^+$}.
\]
Finally, take 
\[
\Omega^\ast=\cap_{\epsilon\in \mathbb{Q}^+} \Omega_\epsilon^+ \cap \cap_{\epsilon\in \mathbb{Q}^+} \Omega_\epsilon^-.
\]
This is an a.s. event, and we have 
\[
\limsup_{n\to\infty} \frac{\log|x(n)|}{\log n}=\frac{1}{\alpha}, \quad\text{on $\Omega^\ast$}.
\]
Hence
\[
\limsup_{n\to\infty} \frac{\log|x(n)|}{\log n}=\frac{1}{\alpha}, \quad\text{a.s.}
\]
A similar analysis of the positive and negative fluctuations leads to 
\[
\limsup_{n\to\infty} \frac{\log x(n)}{\log n}=\frac{1}{\alpha}, \quad
\limsup_{n\to\infty} \frac{\log (-x(n))}{\log n}=\frac{1}{\alpha}, \quad 
\text{a.s.}
\]
\end{example}

\subsection{Time Averages} 
The main theme of the results we have presented is that the properties of the forcing sequence $H$ are reflected in the solution $x$ of 
equation \eqref{eq.x}. So far, we have concentrated on the boundedness or unboundedness of solutions, the size of fluctuations of solutions, the growth rate of solutions, and the times at which the forcing sequence and solution reach record maxima. In this final section we explore one further connection between $H$ and $x$, which does not relate to the \emph{pointwise} size of $H$ and $x$, but rather their \emph{average} values. This is of particular interest in the case that $H$ is a stochastic sequence, because such sequences can be unbounded, but can have finite time averages. 

Very roughly, our most general result states that if $\varphi$ is an increasing convex function, then the finiteness of the ``$\varphi$''--moments
\[
\limsup_{n\to\infty} \frac{1}{n}\sum_{j=1}^n \varphi(|H(j)|)<+\infty \text{ and }
\limsup_{n\to\infty} \frac{1}{n}\sum_{j=1}^n \varphi(|x(j)|)<+\infty 
\]
are equivalent, modulo some small adjustments inside the argument of $\varphi$. In the important case that $\varphi(x)=x^p$ for $p\geq 1$, these small adjustments are unnecessary, and we have that 
\[
\limsup_{n\to\infty} \frac{1}{n}\sum_{j=1}^n |H(j)|^p<+\infty \text{ if and only if }
\limsup_{n\to\infty} \frac{1}{n}\sum_{j=1}^n |x(j)|^p<+\infty, 
\] 
so that the $p$--th moment of $x$ is finite if and only if the $p$--the moment of $H$ is. The equivalence of the finiteness of the $\varphi$--moments also holds in the more general case that $\varphi$ is a regularly varying function at infinity.

In order to make our discussion precise, we recall the definition of convexity of a real function, and a discrete variant of an important inequality relating to convex functions, namely Jensen's inequality. 
\theoremstyle{definition}
\begin{definition}
Let $I$ be a convex set in $\mathbb{R}$ and let $\varphi:I\to\mathbb{R}$. Then $\varphi$ is convex on $I$ if and only if 
\begin{equation*} \varphi(tx_1+(1-t)x_2)\leq t\varphi(x_1)+(1-t)\varphi(x_2) \quad \text{for all $x_1,x_2 \in I$ and all $t\in \left[0,1\right]$}.
\end{equation*}
\end{definition}
\begin {lemma}[Jensen's Inequality] \label{lemma.Jensen}
If $0\leq a_1,a_2,...a_n$ are such that $\sum_{i=1}^{n} a_i=1$ and if $\varphi$ is a convex function, then
\begin{equation*} \varphi\left(\sum_{i=1}^{n} a_ix_i\right)\leq \sum_{i=1}^{n} a_i\varphi(x_i). \end{equation*}
\end{lemma}
We state next our main result: its proof is in the last section of the paper. 
\begin {theorem} \label{theorem.ergodic}
Suppose $f$ obeys \eqref{eq.f1} and \eqref{eq.f2}, and $k$ obeys \eqref{eq.kl1}.
Let $x$ be the solution of \eqref{eq.x} and $\varphi:[0,\infty)\to [0,\infty)$ be an increasing convex function.
\begin{itemize}
\item[(i)]
If there exists $\eta>0$ such that 
\begin{equation*} 
\limsup_{n\to\infty} \frac{1}{n} \sum_{j=1}^{n} \varphi\big((1+\eta)\left|H(j)\right|\big)<+\infty,
\end{equation*}
then 
\begin{equation*} \limsup_{n\to\infty} \frac{1}{n} \sum_{j=0}^{n} \varphi\big(\left|x(j)\right|\big)<+\infty. 
\end{equation*}
\item[(ii)]
If there exists $\eta>0$ such that 
\begin{equation*} \limsup_{n\to\infty} \frac{1}{n} \sum_{j=0}^{n} \varphi\big((1+\eta)\left|x(j)\right|\big)<+\infty,
\end{equation*}
then 
\begin{equation*} \limsup_{n\to\infty} \frac{1}{n} \sum_{j=1}^{n} \varphi\big(\left|H(j)\right|\big)<+\infty. 
\end{equation*}
\end{itemize}
\end{theorem}

It can be seen that we have nearly shown the equivalence of the finiteness of 
\begin{equation*} 
\limsup_{n\to\infty} \frac{1}{n} \sum_{j=1}^{n} \varphi\big(\left|H(j)\right|\big) 
\text{ and } 
\limsup_{n\to\infty} \frac{1}{n} \sum_{j=1}^{n} \varphi\big(\left|x(j)\right|\big). 
\end{equation*}
However, when applying Jensen's inequality to estimate the sums, it is necessary to impose a slightly stronger summability hypothesis on $H$ in order to get the finiteness of the ``$\varphi$''--moment of $x$. In the case when $\varphi(x)=x^p$ (or more generally when $\varphi$ is a convex and regularly varying function at infinity (see e.g.~\cite{BGT})) we can forego this slight restriction, and show that the existence of the $\varphi$--moments of $H$ and $x$ are equivalent. This result is of particular interest if $H$ is a stationary stochastic process, for it shows that the only way in which $x$ will have a finite $p$--th moment is if $H$ does also. This also enables us to make predictions about so--called moment explosion: if, for some $p$, 
\[
\limsup_{n\to\infty} \frac{1}{n}\sum_{j=1}^n |H(j)|^p = +\infty,
\]
then it is automatically true by the next result, that 
\[
\limsup_{n\to\infty} \frac{1}{n}\sum_{j=1}^n |x(j)|^p = +\infty.
\]
\begin{theorem} \label{theorem.pthmoment}
Suppose $f$ obeys \eqref{eq.f1} and \eqref{eq.f2}, and $k$ obeys \eqref{eq.kl1}. Let $x$ be the solution of \eqref{eq.x} and $p\geq 1$. Then the following are equivalent:
\begin{itemize}
\item[(i)] 
\[
\limsup_{n\to\infty} \frac{1}{n} \sum_{j=1}^{n} \left|H(j)\right|^p <+\infty;
\]
\item[(ii)] 
\[
\limsup_{n\to\infty} \frac{1}{n} \sum_{j=1}^{n} \left|x(j)\right|^p <+\infty;
\]
\end{itemize}
\end{theorem} 
\begin{proof}
Let $\varphi(x)=x^p$. Since $p\geq 1$, $\varphi$ is an increasing convex function from $[0,\infty)$ to $[0,\infty)$. Suppose that (i) is true i.e.,
\[
\limsup_{n\to\infty} \frac{1}{n} \sum_{j=1}^{n} \left|H(j)\right|^p <+\infty.
\]
Therefore, because $\varphi((1+\eta)x)=(1+\eta)^p \varphi(x)$ for every $x\geq 0$ and $\eta>0$, we have 
\begin{align*}
\limsup_{n\to\infty} \frac{1}{n} \sum_{j=1}^{n} \varphi((1+\eta)\left|H(j)\right|)
&=
\limsup_{n\to\infty} \frac{1}{n} \sum_{j=1}^{n} (1+\eta)^p\left|H(j)\right|^p \\
&=(1+\eta)^p\limsup_{n\to\infty} \frac{1}{n} \sum_{j=1}^{n} \left|H(j)\right|^p 
<+\infty.
\end{align*}
Therefore, by part (i) of Theorem \ref{theorem.ergodic}, it follows that 
\[
\limsup_{n\to\infty} \frac{1}{n} \sum_{j=1}^{n} \left|x(j)\right|^p=
\limsup_{n\to\infty} \frac{1}{n} \sum_{j=1}^{n} \varphi(\left|x(j)\right|) <+\infty,
\]
which is (ii). Conversely, suppose (ii) holds. Again, using $\varphi((1+\eta)x)=(1+\eta)^p \varphi(x)$ for every $x\geq 0$ and $\eta>0$, we get
\begin{align*}
\limsup_{n\to\infty} \frac{1}{n} \sum_{j=1}^{n} \varphi((1+\eta)\left|x(j)\right|)
&=
\limsup_{n\to\infty} \frac{1}{n} \sum_{j=1}^{n} (1+\eta)^p\left|x(j)\right|^p \\
&= (1+\eta)^p\limsup_{n\to\infty} \frac{1}{n} \sum_{j=1}^{n}\left|x(j)\right|^p 
<+\infty.
\end{align*}
Therefore, by part (ii) of Theorem \ref{theorem.ergodic}, it follows that 
\[
\limsup_{n\to\infty} \frac{1}{n} \sum_{j=1}^{n} \left|H(j)\right|^p=
\limsup_{n\to\infty} \frac{1}{n} \sum_{j=1}^{n} \varphi(\left|H(j)\right|) <+\infty,
\]
which is (i). Hence we have shown that the statements (i) and (ii) are equivalent, as claimed. 
\end{proof}

\subsection{Applications to stochastic processes}
We suppose as earlier that $H$ is a sequence of independent and identically distributed random variables with distribution 
function $F$ and support on $\mathbb{R}$. Then $\varphi(|H(n)|)$ has a finite mean if and only if 
\[
\int_{x\in \mathbb{R}} \varphi(|x|)\,dF(x)<+\infty.
\]
If this is the case, by the strong law of large numbers, we have that 
\[
\lim_{n\to\infty} \frac{1}{n}\sum_{j=1}^n \varphi(|H(j)|)= \int_{x\in \mathbb{R}} \varphi(|x|)\,dF(x), \quad \text{a.s.}
\]
In fact, it is even true when 
\[
\int_{x\in \mathbb{R}} \varphi(|x|)\,dF(x)=+\infty
\]
that 
\[
\lim_{n\to\infty} \frac{1}{n}\sum_{j=1}^n \varphi(|H(j)|)= +\infty, \quad \text{a.s.}
\]
It is then a matter of checking whether we can introduce the small parameter $\eta>0$ into the argument of $\varphi$ in order to apply 
Theorem~\ref{theorem.ergodic}. We show now, by re--examining the heavy--tailed and Gaussian examples studied earlier, that this can be 
achieved with relative success in a number of situations.

\begin{example} \label{examp.heavytail2}
In Example~\ref{examp.heavytail}, we have that $F(x)\sim c_1|x|^{-\alpha}$ as $x\to-\infty$ and $1-F(x)\sim c_2x^{-\alpha}$ as $x\to\infty$. Suppose that $\varphi(x)=x^p$. Then $p<\alpha$ implies 
\[
\int_{x\in \mathbb{R}} \varphi(|x|)\,dF(x)<+\infty
\]
and therefore 
\[
\lim_{n\to\infty} \frac{1}{n}\sum_{j=1}^n |H(j)|^p = \int_{x\in \mathbb{R}} |x|^p\,dF(x), \quad\text{a.s.}
\]
In the case that $p\geq \alpha$ 
\[
\lim_{n\to\infty} \frac{1}{n}\sum_{j=1}^n |H(j)|^p =+\infty, \quad \text{a.s.}
\]
Hence by Theorem~\ref{theorem.ergodic} and Theorem~\ref{theorem.pthmoment} we have a.s.
\[
\limsup_{n\to\infty} \frac{1}{n}\sum_{j=1}^n |x(j)|^p =\left\{
\begin{array}{cc}
+\infty, &  \text{if $p\geq \alpha$}, \\
\in (0,\infty), & \text{if $1\leq p<\alpha$.}
\end{array}
\right.
\]
\end{example}

\begin{example} \label{examp.normal2}
In Example~\ref{examp.normal}, we have that the density of the normal is given by 
\[
g(x)=\frac{1}{\sigma\sqrt{2\pi}}e^{-x^2/(2\sigma^2)}, \quad x\in \mathbb{R}.
\]
Take $\varphi(x)=e^{ax^2}$ for $a>0$. Then $\varphi'(x)=2ax\varphi(x)$ and $\varphi''(x)=2ax\varphi'(x)+2a\varphi(x)>0$. Hence 
$\varphi$ is increasing and convex. Moreover for any $\eta>0$ we have 
\[
\int_{\mathbb{R}} \varphi((1+\eta)|x|)dF(x)=  \frac{1}{\sigma\sqrt{2\pi}} \int_{\mathbb{R}} e^{x^2(a(1+\eta)-1/(2\sigma^2))}\,dx. 
\]
The integral is finite if $a(1+\eta)<1/(2\sigma^2)$ and infinite if $a(1+\eta)\geq 1/(2\sigma^2)$. Thus for $a<1/(2\sigma^2)$ we can choose $\eta>0$ sufficiently small such that $a(1+\eta)<1/(2\sigma^2)$, and so by the strong law
\[
\lim_{n\to\infty} \frac{1}{n}\sum_{j=1}^n \varphi((1+\eta)|H(j)|) \text{ is finite a.s.}
\]
Therefore
\[
\limsup_{n\to\infty} \frac{1}{n}\sum_{j=1}^n e^{ax(j)^2} \text{ is finite a.s. if $a<1/(2\sigma^2)$}.
\]
On the other hand, suppose that $a> 1/(2\sigma^2)$ is fixed and that there is an event $A$ of positive probability 
such that 
\[
A=\{\omega: \limsup_{n\to\infty} \frac{1}{n}\sum_{j=1}^n e^{ax(j)^2}(\omega) < +\infty\}.
\]
There is $\eta>0$ such that $b:=a/(1+\eta)>1/(2\sigma^2)$. Define $\phi(x)=e^{bx^2}$. Then for $\omega\in A$ we have 
$\limsup_{n\to\infty} n^{-1}\sum_{j=1}^n \phi((1+\eta)|x(j)|)(\omega)<+\infty$. This implies by part (ii) of 
Theorem~\ref{theorem.ergodic} that 
\[
\limsup_{n\to\infty} \frac{1}{n}\sum_{j=1}^n \phi(|H(j)|)(\omega)=\limsup_{n\to\infty} \frac{1}{n}\sum_{j=1}^n e^{bH(j)^2}(\omega)<+\infty.
\]
On the other hand, since $b>1/(2\sigma^2)$, we have that 
\[
\int_{\mathbb{R}} \phi(|x|)dF(x)=+\infty, 
\]
so it follows that 
\[
\lim_{n\to\infty} \frac{1}{n}\sum_{j=1}^n \phi(|H(j)|)=+\infty, \quad \text{a.s.}
\]
This contradicts the assumption that $A$ is an event of positive probability. Therefore, we must have $\mathbb{P}[A]=0$ and so 
\[
\limsup_{n\to\infty} \frac{1}{n}\sum_{j=1}^n e^{ax(j)^2} = +\infty, \quad\text{a.s. if $a>1/(2\sigma^2)$}.
\]
Summarising our conclusions we have that, a.s.
\[
\limsup_{n\to\infty} \frac{1}{n}\sum_{j=1}^n e^{ax(j)^2}
=\left\{
\begin{array}{cc}
+\infty, &  \text{if $a> 1/(2\sigma^2)$}, \\
\in (0,\infty), & \text{if $a<1/(2\sigma^2)$.}
\end{array}
\right.
\]
Our analysis is not sufficiently refined to conclude what the situation is if $a=1/(2\sigma^2)$.
\end{example}

\subsection{Further work}
Scrutiny of the proofs that follow shows that the analysis presented here for \eqref{eq.x} works with trivial modification for the 
continuous--time integral equation 
\begin{equation} \label{eq.xcns}
x(t)=H(t)+\int_0^t k(t-s)f(x(s))\,ds, \quad t\geq 0.
\end{equation}
It is still assumed that $f$ obeys \eqref{eq.f1} and \eqref{eq.f2}. We assume now that $k$ is in $L^1(0,\infty)$. For continuous solutions, we ask that $k$ and $H$ are continuous, and to guarantee uniqueness of a continuous solution of \eqref{eq.xcns}, we can assume 
that $f$ is locally Lipschitz continuous. Then direct analogues of all the main results apply. 

We have not focussed on nonconvolution equations, but it is easy to see that the proofs of all results (with the possible exception of 
Theorem~\ref{theorem.ergodic}) go through with cosmetic changes for the Volterra summation equation  
\begin{equation} \label{eq.xnoncon}
x(n+1)=H(n+1)+\sum_{j=0}^n k(n,j) f(x(j)),  \quad n\geq 0; \quad x(0)=\xi,
\end{equation}
where $k:\mathbb{Z}^+\times \mathbb{Z}^+\to \mathbb{R}$ is such that 
\[
\sup_{n\geq 0} \sum_{j=0}^n |k(n,j)|<+\infty,
\]
and $f$ once again obeys \eqref{eq.f1} and \eqref{eq.f2}. The corresponding nonconvolution integral equation
\[
x(t)=H(t)+\int_0^t k(t,s)f(x(s))\,ds, \quad t\geq 0,
\]
can also be analysed successfully, once $\sup_{t\geq 0} \int_0^t |k(t,s)|\,ds <+\infty$. 

We have remarked already that many interesting results of a similar character to those presented here can be obtained for 
the linear Volterra summation equation
\begin{equation} \label{eq.xlinear}
x(n+1)=H(n+1)+\sum_{j=0}^n k(n-j)x(j), \quad n\geq 0; \quad x(0)=\xi
\end{equation}
and the corresponding linear integral equation 
\[
x(t)=H(t)+\int_0^t k(t-s)x(s)\,ds, \quad t\geq 0.
\]
There are two chief differences in the nature of the results: first, the manner in which the kernel $k$ fades is important in the linear case, in stark contrast to the situation here. The reader will have seen throughout how small a role $k$ plays in the nature of the solution $x$, whose properties are inherited rather directly from $H$: there is no ``long memory'' or hysteresis effect present in \eqref{eq.x}, which can contrast markedly with the situation in \eqref{eq.xlinear} in the case when $k$ fades slowly. 
Second, the type of ``nice'' unbounded space we consider in the linear case tends to be slightly more restrictive than that we consider here, mainly because the Volterra term in \eqref{eq.xlinear} can be of the same order as $x(n+1)$ when the latter is large. On the other hand, the corresponding Volterra term in \eqref{eq.x} is of smaller order when $x(n+1)$ is large: this makes the analysis considerably easier, and therefore weaker hypotheses on the data suffice to make good progress in the sublinear case.       

Analysis of systems in $\mathbb{R}^p$ requires more thought. From the perspective of applications, it is of evident interest to study 
not only $\max_{0\leq j\leq n} \|x(j)\|$ (where $\|\cdot\|$ is a norm on $\mathbb{R}^p$), but also the running maximum of the $i$--th component of the system $\max_{0\leq j\leq n} |x_i(j)|$ ($i=1,\ldots, p$). Such an analysis likely requires a more 
delicate analysis of the maxima than the confines of this paper allow. 

\section{Proof of Theorem~\ref{prop.bounded}} 
The proof is elementary, but several useful estimates are developed which we employ in later proofs. Therefore, we 
give more intermediate details than are strictly necessary for present purposes. The proof takes its inspiration in 
part from \cite[Lemma 5.3]{jagydr06}. 
 
Notice that 
the hypotheses \eqref{eq.f1} and \eqref{eq.f2} on $f$ imply
\begin{equation} \label{eq.fast}
\text{For every $\epsilon>0$ there is $F(\epsilon)>0$ such that }
\left|f(x)\right|\leq F(\epsilon) + \epsilon|x|, \quad \forall x\in \mathbb{R}. 
\end{equation} 
Taking absolute values across \eqref{eq.x} and using the inequality \eqref{eq.fast} above, gives
\begin{align} 
\left|x(n+1)\right|&\leq\left|H(n+1)\right|+\sum^{n}_{l=0}\left|k(l)\right| F(\epsilon)+\epsilon \sum^{n}_{j=0}\left|k(n-j)
\right|\left|x(j)\right|\nonumber\\
& \leq \left|H(n+1)\right|+\left|k\right|_{1}F(\epsilon)+\epsilon\sum^{n}_{j=0}\left|k(n-j)\right|\left|x(j)\right|\nonumber\\
& \leq H^*(n+1)+\left|k\right|_1F(\epsilon)+\epsilon\sum^{n}_{j=0}\left|k(n-j)\right|x^*(n) \nonumber\\
\label{eq.xestsumprop1}
& \leq H^*(n+1)+\left|k\right|_1F(\epsilon)+\epsilon |k|_1 x^*(n). 
\end{align}
Next let  $N\geq0$ and take the maximum over $n=0$ to $N$ on both sides of \eqref{eq.xestsumprop1} to get 
\begin{align}
\max_{{1}\leq{j}\leq{N+1}}\left|x(j)\right|&\leq \max_{{0}\leq{n}\leq{N}} \left\{H^*(n+1)\right\}+\left|k\right|_1F(\epsilon)+\max_{{0}\leq{n}\leq{N}}\left\{\epsilon\left|k\right|_1x^*(n)\right\} \nonumber \\
&=H^*(N+1)+\left|k\right|_1F(\epsilon)+\epsilon\left|k\right|_1 \max_{{0}\leq{n}\leq{N}} x^*(n) \nonumber \\
\label{eq.xastprop1}
&=H^*(N+1)+\left|k\right|_1F(\epsilon)+\epsilon\left|k\right|_1 x^*(N).
\end{align}
Next observe that 
\begin{align*} x^*(N+1)&=\max\left\{\left|x(0)\right|, \max_{{1}\leq{j}\leq{N+1}}\left|x(j)\right|\right\}
\leq \left|x(0)\right|+\max_{{1}\leq{j}\leq{N+1}}\left|x(j)\right|.
\end{align*}
Since $x^*$ is a non-decreasing sequence, we have from the last inequality and \eqref{eq.xastprop1} that
\begin{align*} x^*(N+1)&\leq\left|x(0)\right|+\max_{{1}\leq{j}\leq{N+1}}\left|x(j)\right|\\
&\leq\left|k\right|_1F(\epsilon)+\left|x(0)\right|+H^*(N+1)+\epsilon\left|k\right|_1x^*(N)\\
&\leq\left|k\right|_1F(\epsilon)+\left|x(0)\right|+H^*(N+1)+\epsilon\left|k\right|_1x^*(N+1).
\end{align*}
Now, let $\epsilon>0$ be so small that $\epsilon\left|k\right|_1<1$. Then for all $N\geq 0$
\begin{equation*}(1-\epsilon\left|k\right|_1)x^*(N+1)\leq\left|x(0)\right|+\left|k\right|_1F(\epsilon)+H^*(N+1).
\end{equation*}
By construction, $1-\epsilon\left|k\right|_1>0$. Hence 
\begin{equation}  \label{eq.keyupperestimate}
0\leq x^*(n)\leq \frac{1}{1-\epsilon\left|k\right|_1}\biggl(\left|x(0)\right|+\left|k\right|_1F(\epsilon)+H^*(n)\biggr), \quad n\geq 1.
\end{equation}

Assume now that $H^*(n)\to H_\infty <+\infty$ as $n\to\infty$. Then
\begin{align*} 
\limsup_{n\to +\infty}x^*(n)&\leq\frac{1}{1-\epsilon\left|k\right|_1}\biggl(\left|x(0)\right|+\left|k\right|_1F(\epsilon)
+H_\infty\biggr)<+\infty.
\end{align*}
Therefore $x^*$ is a positive sequence which is bounded above. Moreover, it is non-decreasing and thus 
when $n\to\infty$, we have 
\[
x^*(n)\to x_\infty \leq \frac{1}{1-\epsilon\left|k\right|_1}\biggl(\left|x(0)\right|+\left|k\right|_1F(\epsilon)+H_\infty\biggr).
\]
Hence $x$ is bounded, which proves part (a). 

Now we turn to the proof of part (b) of the result, wherein we assume $H^*(n)\to\infty$ as $n\to\infty$. We suppose that $x^*(n)\to x_\infty<+\infty$ as $n\to\infty$ and see that this leads to a contradiction.
Define
\begin{equation} \label{eq.Sn}
S(n):=\sum^{n}_{j=0}k(n-j)f(x(j)).
\end{equation}
Proceeding as in the estimate of \eqref{eq.xestsumprop1} in part (a), we get
\begin{equation} \label{eq.Sest}
\left|S(n)\right|\leq\left|k\right|_1F(\epsilon)+\epsilon\left|k\right|_1x^*(n), \quad n\geq 0.
\end{equation}
Rearranging \eqref{eq.x} gives $H(n+1)=x(n+1)-S(n)$. 
Therefore by \eqref{eq.Sest} 
\begin{align*} \left|H(n+1)\right|
&\leq\left|x(n+1)\right|+\left|k\right|_1F(\epsilon)+\epsilon\left|k\right|_1x^*(n)\\
&\leq x^*(n+1)+\left|k\right|_1F(\epsilon)+\epsilon\left|k\right|_1x^*(n+1).
\end{align*}
Hence 
\begin{equation} \label{eq.keyestimate2}
\left|H(n)\right| \leq (1+\epsilon\left|k\right|_1)x^*(n)+\left|k\right|_1F(\epsilon), \quad n\geq 1.
\end{equation}
Moreover from this we get
\begin{equation} \label{eq.Hastupper}
H^\ast(n) \leq (1+\epsilon\left|k\right|_1)x^*(n)+\left|k\right|_1F(\epsilon), \quad n\geq 1.
\end{equation}
We have assumed that $x^*(n)\to x_\infty<+\infty$ as $n\to\infty$, and therefore
\begin{align*} 
\limsup_{n \to +\infty}\left|H(n)\right|&\leq1+\epsilon\left|k\right|_1 x_\infty+\left|k\right|_1F(\epsilon) <+\infty.
\end{align*}
This means that $\left|H\right|$ is bounded which in turn means that $H^*(n)\to H_\infty<+\infty$ as $n\to\infty$, which gives the desired contradiction, completing the proof of part (b). 

\section{Proof of Theorem~\ref{theorem.growthmaxima}} 
Suppose that $\epsilon\in (0,1)$ is so small that $\epsilon|k|_1<1$, then there is $F(\epsilon)>0$ such that $f$ obeys \eqref{eq.fast}. As in the proof of Theorem~\ref{prop.bounded}, we have the estimate \eqref{eq.keyupperestimate}.
Therefore, as $H^\ast(n)\to\infty$ as $n\to\infty$, we get 
\begin{equation*}
\limsup_{n \to +\infty}\frac{x^*(n)}{H^*(n)}\leq\frac{1}{1-\epsilon\left|k\right|_1}.
\end{equation*}
Letting $\epsilon\to0^+$ yields
\begin{equation} \label{eq.xHlimsup}
\limsup_{n \to +\infty}\frac{x^*(n)}{H^*(n)}\leq1.
\end{equation}

This gives the required upper estimate in part (i). To get the lower estimate (i.e., to obtain a lower bound on 
$\liminf_{n \to +\infty} x^*(n)/H^*(n)$), we start by recalling the estimate \eqref{eq.Hastupper}
which rearranges to give 
\begin{align*}
x^*(n)\geq 
\frac{1}{1+\epsilon\left|k\right|_1}H^\ast(n) - \frac{\left|k\right|_1F(\epsilon)}{1+\epsilon\left|k\right|_1}.
\end{align*}
Since $H^*(n)\to\infty$ as $n\to\infty$ by hypothesis, taking the limit inferior as $n\to\infty$ yields
\begin{equation*} 
\liminf_{n \to +\infty}\frac{x^*(n)}{H^*(n)}\geq\frac{1}{1+\epsilon\left|k\right|_1}, 
\end{equation*}
and now letting $\epsilon\to0^+$ yields 
\begin{equation*} 
\liminf_{n \to +\infty}\frac{x^*(n)}{H^*(n)}\geq1. 
\end{equation*}
Combining this with \eqref{eq.xHlimsup} gives part (i).  

We next prove part (iii) of the result. By definition of $x$ and $S$ in \eqref{eq.Sn}, if $t^x_n\geq1$, we have
$ x(t^x_n)=H(t^x_n)+S(t^x_n-1)$. 
Thus by \eqref{eq.Sest}, 
$\left|x(t^x_n)\right|\leq\left|H(t^x_n)\right|+F(\epsilon)\left|k\right|_1+\epsilon\left|k\right|_1x^*(t^x_n-1)$. Now, by the monotonicity of $x^\ast$ and the definition of $t_n^x$ we have 
\begin{align*} 
x^*(t^x_n-1)&\leq x^*(t^x_n) =\max_{{0}\leq{j}\leq{t^x_n}} \left|x(j)\right| =\left|x(t^x_n)\right|,
\end{align*}
since $t^x_n\leq n$ and $\left|x(t^x_n)\right|=\max_{{0}\leq{j}\leq{n}}\left|x(j)\right|$. Therefore,
\begin{equation*} 
\left|x(t^x_n)\right|\leq\left|H(t^x_n)\right|+F(\epsilon)\left|k\right|_1+\epsilon\left|k\right|_1|x(t^x_n)|. 
\end{equation*}
If $H^*(n)\to\infty$ as $n\to\infty$, we have $\max_{{0}\leq{j}\leq{n}}\left|x(j)\right|\to\infty$ as 
$n\to\infty$ and so $\left|x(t^x_n)\right|\to\infty$ as $n\to\infty$. The above inequality then implies 
that $\left|H(t^x_n)\right|\to\infty$ as $n\to\infty$. Rearranging and taking limits as before yields
\begin{equation}\label{eq.xtnxHtnx}
\limsup_{n \to +\infty}\frac{\left|x(t^x_n)\right|}{\left|H(t^x_n)\right|}\leq 1. 
\end{equation}
We now get a lower estimate for the limit. First, rearranging \eqref{eq.x} at the time $t_n^x$ and taking the  triangle 
inequality and the estimate \eqref{eq.fast} gives 
\begin{align*} \left|H(t^x_n)\right|& \leq \left|x(t^x_n)\right|+F(\epsilon)\left|k\right|_1+\epsilon\sum^{t^x_n-1}_{j=0}\left|k(t^x_n-1-j)\right||x(j)|.
\end{align*}
Now, for $j=0,...,t^x_n-1$,
\begin{align*} \left|x(j)\right|&\leq \max_{{0}\leq{j}\leq{t^x_n-1}}\left|x(j)\right|
\leq \max_{{0}\leq{j}\leq{t^x_n}}\left|x(j)\right| =\left|x(t^x_n)\right|.
\end{align*}
Hence
$\left|H(t^x_n)\right|\leq\left|x(t^x_n)\right|\left(1+\epsilon\left|k\right|_1\right)+F(\epsilon)\left|k\right|_1$.
Rearranging this inequality and taking limits gives
\begin{equation*} 
\liminf_{n \to +\infty}\frac{\left|x(t^x_n)\right|}{\left|H(t^x_n)\right|}
\geq\frac{1}{1+\epsilon\left|k\right|_1}.
\end{equation*}
Letting $\epsilon\to0^+$ and combining with \eqref{eq.xtnxHtnx} yields $|x(t^x_n)|/|H(t^x_n)|\to1$ as $n\to\infty$,
completing the proof of the first limit in part (iii).

We now prove the second limit in part (iii), namely $\lim_{n \to +\infty} \left|H(t^x_n)\right|/\left|H(t^H_n)\right|=1$. 
Notice that part (i) of this Theorem gives $x^*(n)/H^*(n)\to1$ as $n\to\infty$. By definition, 
$x^*(n)=\left|x(t^x_n)\right|$ and $H^*(n)=\left|H(t^H_n)\right|$. It has just been shown that 
$|x(t^x_n)|/|H(t^x_n)|\to1$ as $n\to\infty$. Therefore, as $n\to\infty$,
\begin{align*} \frac{\left|H(t^x_n)\right|}{\left|H(t^H_n)\right|}
&=\frac{\left|H(t^x_n)\right|}{\left|x(t^x_n)\right|}\cdot\frac{\left|x(t^x_n)\right|}{H^*(n)}
=\frac{\left|H(t^x_n)\right|}{\left|x(t^x_n)\right|}\cdot\frac{x^*(n)}{H^*(n)}\to 1.
\end{align*}
This proves the second limit in part (iii).

Finally, we prove part (ii). By assumption, $H^*(n)\to\infty$ as $n\to\infty$. Since $\left|H(t^H_n)\right|=\max_{{1}\leq{j}\leq{t^H_n}} \left|H(j)\right|=H^\ast(t_n^H)$, we have
\begin{equation} \label{eq.xtnh/htnh} 
\lim_{n \to +\infty}\frac{x^*(t^H_n)}{\left|H(t^H_n)\right|}
=\lim_{n\to\infty} \frac{x^\ast(t_n^H)}{H^\ast(t_n^H)} =1,
\end{equation} 
using part (i). Now, $|x(t^H_n)|\leq\max_{{0}\leq{j}\leq{t^H_n}} |x(j)|$, so we have
\begin{equation} \label{eq.xtnHHtnHsup} \limsup_{n \to +\infty}\frac{|x(t^H_n)|}{\left|H(t^H_n)\right|}\leq1. \end{equation}
From \eqref{eq.Sest} we get 
$
|S(t^H_n-1)|\leq|k|_1F(\epsilon)+\epsilon\left|k\right|_1x^*(t^H_n-1)
\leq \left|k\right|_1F(\epsilon)+\epsilon\left|k\right|_1x^*(t^H_n)$. 
Since $x(t^H_n)=H(t^H_n)+S(t^H_n-1)$, we have
\begin{align*} |H(t^H_n)|
&\leq |x(t^H_n)|+|S(t^H_n-1)|
\leq |x(t^H_n)|+\left|k\right|_1F(\epsilon)+\epsilon\left|k\right|_1x^*(t^H_n). 
\end{align*}
Therefore,
\begin{equation*} \frac{|x(t^H_n)|}{|H(t^H_n)|}
\geq 1-\frac{\left|k\right|_1F(\epsilon)}{|H(t^H_n)|}
-\frac{\epsilon\left|k\right|_1 x^*(t^H_n)}{|H(t^H_n)|}, 
\end{equation*}
and thus by \eqref{eq.xtnh/htnh}, we have $\liminf_{n \to +\infty} |x(t^H_n)|/|H(t^H_n)|\geq1-\epsilon\left|k\right|_1$. Letting $\epsilon\to 0^+$ yields $\liminf_{n \to +\infty}|x(t^H_n)|/|H(t^H_n)|\geq1$ and therefore combining this with \eqref{eq.xtnHHtnHsup} yields the desired, first limit  $\lim_{n \to +\infty} \left|x(t^H_n)\right|/\left|H(t^H_n)\right|=1$ in part (ii). 
We now prove the second limit in part (ii). We have
from part (i) that $x^*(n)/H^*(n)\to 1$ as $n\to\infty$,  
and by the first part of (ii) we get 
 $\lim_{n\to\infty} \left|H(t^H_n)\right|/\left|x(t^H_n)\right|=1$. Therefore,  
\begin{equation*}
\lim_{n\to\infty}\frac{\left|x(t^x_n)\right|}{\left|x(t^H_n)\right|}
=\lim_{n\to\infty}
\frac{x^*(n)}{H^*(n)}\cdot\frac{H^*(n)}{\left|x(t^H_n)\right|}
=\lim_{n\to\infty}\frac{x^*(n)}{H^*(n)} \cdot\frac{\left|H(t^H_n)\right|}{\left|x(t^H_n)\right|}
=1, \end{equation*}
as claimed. 

\section{Proof of Theorems~\ref{theorem.growth} and~\ref{theorem.growth2}}
\subsection{A preparatory lemma}
Before we prove our main result, we need the following preparatory lemma.
\begin{lemma}\label{preplemma}
Let $(a(n))_{n\geq0}$ be a sequence such that 
\[
 a^*(n):=\max_{{0}\leq{j}\leq{n}}\left|a(j)\right|\to\infty \text{ as $n\to\infty$}
\]
	and suppose	there is a sequence $(\tilde{a}(n))_{n\geq1}$ such that $(\tilde{a}(n))_{n\geq1}$ is increasing with $\tilde{a}(n)\to\infty$  as  $n\to\infty$ and 
	$a(n)\sim\tilde{a}(n)$ 	as $n\to\infty.$
Then 
\begin{equation*} 
\lim_{n\to\infty} \frac{a^*(n)}{\tilde{a}(n)}=1.
\end{equation*}
\end{lemma}
\begin{proof}
Since $a(n)\sim\tilde{a}(n),$ for every $0<\epsilon<1,$ there exists $N(\epsilon)\in\mathbb{N}$ such that 
$0<1-\epsilon<a(n)/\tilde{a}(n)<1+\epsilon$ for all $n\geq N(\epsilon)$. 
Therefore, as $\tilde{a}(n)>0$ for all $n\geq N_1$ with $N_2(\epsilon)=\max(N_1,N(\epsilon)),$ we have that $a(n)>0$  for all $n\geq N_2(\epsilon)$.
Let $n\geq N_2(\epsilon)$. Then
\begin{align*} \max_{{0}\leq{j}\leq{n}}\left|a(j)\right|&=\max\left(\max_{{0}\leq{j}\leq{N_2(\epsilon)-1}}\left|a(j)\right|,\max_{{N_2(\epsilon)}\leq{j}\leq{n}}\left|a(j)\right|\right)\\
&\leq \max_{{0}\leq{j}\leq{N_2(\epsilon)-1}}\left|a(j)\right|+\max_{{N_2(\epsilon)}\leq{j}\leq{n}}\left|a(j)\right|\\
&=\max_{{0}\leq{j}\leq{N_2(\epsilon)-1}}\left|a(j)\right|+\max_{{N_2}\leq{j}\leq{n}}a(j)\\
&\leq \max_{{0}\leq{j}\leq{N_2-1}}\left|a(j)\right|+\max_{{N_2}\leq{j}\leq{n}}(1+\epsilon)\tilde{a}(j)\\
&= a^*(N_2-1)+(1+\epsilon)\tilde{a}(n),
\end{align*}
where we have used the fact that $a(j)>0$ for all $j\geq N_2$ to get the third line and the monotonicity of $\tilde{a}$ at the end. 
Hence, for all $n\geq N_2(\epsilon)$, we have
\begin{equation*} 
\frac{a^*(n)}{\tilde{a}(n)}\leq \frac{a^*(N_2-1)}{\tilde{a}(n)}+(1+\epsilon) 
\end{equation*} 
and thus $\limsup_{n\to\infty} a^*(n)/\tilde{a}(n)\leq 1+\epsilon$. 
Letting $\epsilon\to0^+$, gives
\begin{equation} \label{eq.anlimsup}
\limsup_{n\to\infty}\frac{a^*(n)}{\tilde{a}(n)}\leq 1. 
\end{equation}
To complete the proof, we need a corresponding lower bound. Now, for $j \geq N$ we have $a(j)>(1-\epsilon)\tilde{a}(j)$.
Hence for $n\geq N_2$, 
\begin{align*} 
\max_{{N_2}\leq{j}\leq{n}}a(j)&\geq (1-\epsilon) \max_{{N_2}\leq{j}\leq{n}}\tilde{a}(j)=(1-\epsilon)\tilde{a}(n). 
\end{align*}
Thus for  $n\geq N_2(\epsilon)$,
\begin{equation*} 
a^*(n)=\max\left(a^*(N_2-1),\max_{{N_2}\leq{j}\leq{n}}a(j)\right)
\geq \max\big(a^*(N_2(\epsilon)-1),(1-\epsilon)\tilde{a}(n)\big).
\end{equation*}
Since $\tilde{a}(n)\to\infty$ as $n\to\infty$, there exists $N_3(\epsilon)>0$ such that 
$\tilde{a}(n)>a^*(N_2-1)/(1-\epsilon)$ for  $n\geq N_3(\epsilon)$. 
Let $N_4=\max(N_2,N_3)$ and $n\geq N_4.$ Then $(1-\epsilon)\tilde{a}(n)>a^*(N_2(\epsilon)-1)$ and so,
$a^*(n)\geq \max\big(a^*(N_2-1),(1-\epsilon)\tilde{a}(n)\big) =(1-\epsilon)\tilde{a}(n)$. 
Letting $n\to\infty$, we have $\liminf_{n\to\infty}a^*(n)/\tilde{a}(n)\geq1-\epsilon$ 
and letting $\epsilon\to0^+$ gives
\begin{equation} \label{eq.anliminf}\liminf_{n\to\infty}\frac{a^*(n)}{\tilde{a}(n)}\geq1. \end{equation}
Combining \eqref{eq.anliminf} with \eqref{eq.anlimsup} we get
$a^*(n)/\tilde{a}(n)\to 1$ as $n\to\infty$ as required.
\end{proof}

\subsection{Proof of Theorem~\ref{theorem.growth}} 
We only prove part (i). Suppose statement (a) holds. Then the sequence $H(n)$ is asymptotic to another increasing sequence $\tilde{H}(n)$, where 
$\tilde{H}(n)\to\infty$ as $n\to\infty$ . Then by Lemma~\ref{preplemma}, 
$H^*(n)\sim\tilde{H}(n)\sim H(n)$ as $n\to\infty$. 
Since $H(n)\to\infty$, $H^\ast(n)\to\infty$ as $n\to\infty$, so we have $x^*(n)/H^*(n)\to1$ as $n\to\infty$ 
and so we have 
$x^*(n)/H(n)\to1$ as $n\to\infty$. 
Since $H(n)\to\infty$ as $n\to\infty$, we have from \eqref{eq.Sest} the limit
\[
\limsup_{n\to\infty} \frac{|S(n)|}{H(n)} \leq \epsilon\left|k\right|_1\limsup_{n\to\infty} \frac{x^*(n)}{H(n)}=\epsilon\left|k\right|_1.
\]
Letting $\epsilon\to 0^+$ gives $S(n)/H(n)\to 0$ as $n\to\infty$. 
Next since $\tilde{H}$ is increasing, we have
\begin{equation*} 
\limsup_{n\to\infty}\frac{H(n)}{H(n+1)}
=\limsup_{n\to\infty} \frac{H(n)}{\tilde{H}(n)}\cdot\frac{\tilde{H}(n)}{\tilde{H}(n+1)}\cdot \frac{\tilde{H}(n+1)}{H(n+1)}\leq 1.
\end{equation*} 
Therefore we have
\begin{equation*} 
\frac{|S(n-1)|}{H(n)}=\frac{|S(n-1)|}{H(n-1)}\cdot\frac{H(n-1)}{H(n)}\to0 \text{ as }n\to\infty.
\end{equation*}
Since $x(n)=H(n)+S(n-1)$, we get $x(n)/H(n)\to1$ as $n\to\infty$ as required. Moreover as $H$ is asymptotic to $\tilde{H}$ we have  
$x(n)/\tilde{H}(n)\to1$ as $n\to\infty$, so $x$ is asymptotic to the increasing sequence $\tilde{H}$. This proves statement (b).  

Conversely, suppose $x(n)\to\infty$ and $x$ is asymptotic to an increasing sequence $(\tilde{x}(n))_{n\geq1}$. 
Then by Lemma~\ref{preplemma}, it follows that $x^*(n)\sim\tilde{x}(n)$ and so, as $n\to\infty$, we also have $x^*(n)\sim x(n).$
Therefore, since $x^*(n)\to\infty$ as $n\to\infty$, we have $H^*(n)\to\infty$ as $n\to\infty$. 
Hence $x^*(n)/H^*(n)\to1$ as $n\to\infty$. Since $S$ obeys \eqref{eq.Sest}, we have $|S(n)|/x^*(n)\to0$ as $n\to\infty$. Hence, since $x^\ast(n)\leq x^\ast(n+1)$ we have 
\begin{equation*} \frac{\left|S(n-1)\right|}{x(n)}
=\frac{\left|S(n-1)\right|}{x^*(n-1)}\cdot\frac{x^*(n-1)}{x^*(n)}\to0\text{ as }n\to\infty.
\end{equation*}
But since $x(n)=H(n)+S(n-1)$, we have $H(n)/x(n)\to1$ as $n\to\infty$, which proves part of the desired conclusion. 
Recall that $x$ is asymptotic to the increasing sequence $\tilde{x}$, so $H$ is asymptotic to the increasing sequence $\tilde{x}$ which itself tends to infinity. Therefore $H(n)\to\infty$ as $n\to\infty$. Hence $H$ enjoys all the properties listed in statement (a).
\subsection{Proof of Theorem~\ref{theorem.growth2}}
By hypothesis $\lim_{n\to\infty} |H(n)/a(n)-(\Lambda_a H)(n)|=0$ and $\limsup_{n\to\infty} |(\Lambda_a H)(n)| \in (0,\infty)$. Since $a(n)\to\infty$ as $n\to\infty$ 
it follows that $H^\ast(n)\to\infty$ as $n\to\infty$. Hence by part (a) of Theorem~\ref{theorem.growthmaxima}, we have 
$\lim_{n\to\infty} x^\ast(n)/H^\ast(n)=1$. By part (a) of Lemma~\ref{H(n)/a(n)} $(H^\ast(n)/a(n))_{n\geq 1}$ is a bounded sequence. Hence $(x^\ast(n)/a(n))_{n\geq 1}$ is a bounded sequence. 
By \eqref{eq.Sest} we have for every $\epsilon>0$ that $\left|S(n)\right|\leq\left|k\right|_1F(\epsilon)+\epsilon\left|k\right|_1x^*(n)$.
Therefore 
\begin{equation} \label{eq.Saest}
\limsup_{n\to\infty} \frac{|S(n)|}{a(n)} \leq \epsilon|k|_1 \limsup_{n\to\infty} \frac{x^\ast(n)}{a(n)}.
\end{equation}
Since $\epsilon$ is arbitrary, we have $S(n)/a(n)\to 0$ as $n\to\infty$. Since $a$ is an increasing sequence $S(n)/a(n+1)\to 0$ 
as $n\to\infty$. By the identity
\[
\frac{x(n+1)}{a(n+1)}-(\Lambda_a H)(n+1)= 
\frac{H(n+1)}{a(n+1)}-(\Lambda_a H)(n+1) + \frac{S(n)}{a(n+1)}
\]
it is clear that $\lim_{n\to\infty} |x(n+1)/a(n+1)-(\Lambda_a H)(n+1)|=0$ 
so $x$ is in $B_a$ and we can take $\Lambda_a x=\Lambda_a H$, completing the proof.

Conversely, suppose that $x\in B_a$ and there exists a bounded sequence $\Lambda_a x$ such that 
$\lim_{n\to\infty} |x(n)/a(n) - (\Lambda_a x)(n)|=0$. We also have $\limsup_{n\to\infty} |(\Lambda_a x)(n)|\in (0,\infty)$. Then we have  $\limsup_{n\to\infty} |x(n)|/a(n)<+\infty$, and therefore by Lemma~\ref{H(n)/a(n)} we have that $\limsup_{n\to\infty} x^\ast(n)/a(n)<+\infty$. Then \eqref{eq.Saest} holds again for every $\epsilon>0$ and so $S(n)/a(n)\to 0$ as $n\to\infty$. Since $a$ is increasing  $S(n)/a(n+1)\to 0$ 
as $n\to\infty$. Finally, writing 
\[
\frac{H(n+1)}{a(n+1)}-(\Lambda_a x)(n+1) 
=
\frac{x(n+1)}{a(n+1)}-(\Lambda_a x)(n+1) - \frac{S(n)}{a(n+1)},
\]
it is clear that $\lim_{n\to\infty} |H(n+1)/a(n+1)-(\Lambda_a x)(n+1)|=0$ 
so $H$ is in $B_a$ and we can take $\Lambda_a H=\Lambda_a x$, completing the proof.

\section{Proof of Theorems~\ref{theorem.signedfluc} and~\ref{theorem.signedfluc2}}
\subsection{Proof of Theorem~\ref{theorem.signedfluc}}
From \eqref{eq.Sest} we have
\begin{align*}
S^*(n)&:=\max_{{0}\leq{j}\leq{n}}\left|S(j)\right|
 \leq \max_{{0}\leq{j}\leq{n}} \big\{F(\epsilon)\left|k\right|_1+\epsilon\left|k\right|_1x^*(j)\big\}
=F(\epsilon)\left|k\right|_1+\epsilon\left|k\right|_1 x^\ast(n).
\end{align*}
Therefore $S^\ast(n)/x^*(n)\to0$ as $n\to\infty$. 

Our next task is to deduce a lower estimate for $x^\ast_+$. 
For $n\geq0$, we get
\begin{align*} 
H^*_+(n+1)
&=\max_{0\leq j \leq n }H(j+1)\leq \max_{0\leq j\leq n}\left\{H(j+1)+S(j)+\left|S(j)\right|\right\}\\
&\leq \max_{0\leq j \leq n}\left\{H(j+1)+S(j)\right\}+\max_{{0}\leq{j}\leq{n}}\left|S(j)\right|\\
&=\max_{0\leq j\leq n}x(j+1)+\max_{0\leq j\leq n}\left|S(j)\right|\leq \max_{0\leq j\leq n+1}x(j)+S^*(n)\\
&\leq x^*_+(n+1)+S^*(n+1).
\end{align*}
Therefore by definition, 
\begin{equation} \label{eq.xpluslower}
x^*_+(n)\geq H^*_+(n)-S^*(n), \quad n\geq 1.
\end{equation}
We next obtain a lower estimate for $x^\ast_-$. Since $-x(n+1)+S(n)=-H(n+1)$, we have
\begin{align*}
H^*_-(n+1)
&=\max_{{0}\leq{j}\leq{n}}\big(-x(j+1)+S(j)\big)\leq \max_{{0}\leq{j}\leq{n}} \big(-x(j+1)+\left|S(j)\right|\big)\\
&\leq \max_{{0}\leq{j}\leq{n}} \big(-x(j+1)\big)+\max_{{0}\leq{j}\leq{n}}\left|S(j)\right|
=\max_{{1}\leq{l}\leq{n+1}} \big(-x(l)\big)+S^*(n)\\
&\leq \max_{{0}\leq{j}\leq{n+1}}\big(-x(j)\big)+S^*(n)\\
&\leq x^\ast_-(n+1)+S^*(n)\leq x^*_-(n+1)+S^*(n+1).
\end{align*}
Thus
\begin{equation} \label{eq.xminuslower}
x^\ast_-(n)\geq H^*_-(n)-S^*(n), \quad n\geq 1.
\end{equation}
We now prove part (i). By Theorem~\ref{theorem.growthmaxima}, $x^*(n)/H^*(n)\to 1$ as $n\to\infty$. 
Since $\lambda \in [0,1)$, 
$x^*(n)/H^*_+(n)\to1$ as $n\to\infty$. Clearly $x^*_+(n)\leq x^*(n)$. Therefore 
$\limsup_{n\to\infty} x^*_+(n)/H^*_+(n)\leq1$. 
We now get a lower estimate. Since $S^\ast(n)/x^\ast(n)\to 0$, $x^\ast(n)/H^\ast(n)\to 1$, $H^\ast_+(n)/H^\ast(n)\to 1$  
as $n\to\infty$, we have $S^*(n)/H^*_+(n)\to0$ as $n\to\infty$. Therefore, by 
\eqref{eq.xpluslower} we have the bound $\liminf_{n\to\infty} x^*_+(n)/H^*_+(n)\geq1$. This implies that $x^*_+(n)/H^*_+(n)\to 1$ as $n\to\infty$ as claimed. 
Therefore we have $\lim_{n\to \infty} x^*(n)/H^*_+(n)=1$. This completes the proof of part (i). The proof of part (ii) is symmetric and omitted.

We now prove part (iii). By Theorem~\ref{theorem.growthmaxima}, we have $x^*(n)/H^*(n)\to1$ as $n\to\infty$. Since $H^*_+(n)\sim H^*_-(n)$as $n\to\infty$, we have $H^*(n)\sim H^*_+(n)\sim H^*_-(n)$ as $n\to\infty$. 
Then we have $S^*(n)/H^*_-(n)\to 0$ as $n\to\infty$ because  $x^*(n)/H^*(n)\to 1$, $S^*(n)/x^*(n)\to 0$ and $H^*(n)/H^*_-(n)\to 1$ as $n\to\infty$.
Hence by \eqref{eq.xminuslower} we get
\begin{equation*} 
\liminf_{n\to\infty}\frac{x^*_-(n)}{H^*_-(n)}
\geq \liminf_{n\to\infty}\left\{\frac{H^*_-(n)}{H^*_-(n)}-\frac{S^*(n)}{H^*_-(n)}\right\}=1. 
\end{equation*}
On the other hand, as $x_-^\ast(n)\leq x^\ast(n)$ we get
\begin{align*} 
\limsup_{n\to\infty}\frac{x^*_-(n)}{H^*_-(n)}
&\leq\limsup_{n\to\infty}\left\{\frac{x^*(n)}{H^*(n)}\cdot\frac{H^*(n)}{H^*_-(n)}\right\}=1 
\end{align*}
and therefore $\lim_{n\to\infty} x^*_-(n)/H^*_-(n)=1$. 
We now deal with the asymptotic behaviour of $x^\ast_+$. As $n\to\infty$, we have 
\[
\frac{S^*(n)}{H^*_+(n)}=\frac{S^*(n)}{H^*_-(n)}\cdot\frac{H^*_-(n)}{H^*_+(n)}\to 0.
\]
Therefore from \eqref{eq.xpluslower} the inequality  
\[
\liminf_{n\to\infty}\frac{x^*_+(n)}{H^*_+(n)}\geq\liminf_{n\to\infty}\left\{1-\frac{S^*(n)}{H^*_+(n)}\right\}=1 
\]
results. Finally, since
\begin{equation*} 
\limsup_{n\to\infty}\frac{x^*_+(n)}{H^*_+(n)}\leq \limsup_{n\to\infty}\frac{x^*(n)}{H^*_+(n)}
=\limsup_{n\to\infty}\left\{\frac{x^*(n)}{H^*(n)}\cdot\frac{H^*(n)}{H^*_-(n)}\right\}
=1, \end{equation*}
we have $\lim_{n\to\infty} x^*_+(n)/H^*_+(n)=1$, 
completing the proof of part (iii).

\subsection{Proof of Theorem~\ref{theorem.signedfluc2}}
We start by deducing an auxiliary upper estimate for a quantity related to $x^\ast_-$. 
To do this we start by writing for $n\geq 0$
\begin{align*} 
\max_{1\leq l \leq n+1} (-x(l))
&=\max_{0\leq j\leq n} (-x(j+1))=\max_{0\leq j\leq n} \{-H(j+1)-S(j)\}\\
&\leq \max_{0\leq j\leq n} \{-H(j+1)+|S(j)|\} \leq \max_{0\leq j\leq n} (-H(j+1)) +\max_{0\leq j\leq n} |S(j)| \\
&= H_-^\ast(n+1) + S^\ast(n) \leq H_-^\ast(n+1) + S^\ast(n+1).
\end{align*}
Therefore 
\begin{equation} \label{eq.xminusupper}
\max_{1\leq l \leq n} (-x(l)) \leq H_-^\ast(n) + S^\ast(n), \quad n\geq 1.
\end{equation}
A corresponding auxiliary estimate for $x^\ast_+$ is also needed. By arguing as above we obtain 
\begin{equation} \label{eq.xplusupper}
 \max_{1\leq l \leq n}x(l) \leq H_+^\ast(n) + S^\ast(n), \quad n\geq 1.
\end{equation}

We are now in a position to prove part (i). Start by considering the case that $\lambda\in (0,1)$. Then $H^\ast(n)\sim H^\ast_+(n)\sim H^\ast_-(n)/\lambda$ as $n\to\infty$, and $H^\ast_-(n)\to\infty$ as $n\to\infty$. 
By part (i) of Theorem~\ref{theorem.signedfluc} we have $x_+^\ast(n)/H_+^\ast(n)\to 1$ as $n\to\infty$. Next as $x^\ast(n)/H^\ast(n)\to 1$ as $n\to\infty$ and $S^\ast(n)/x^\ast(n)\to 0$ as $n\to\infty$, we have as $n\to\infty$ that 
\[
\frac{S^\ast(n)}{H_-^\ast(n)}
=
\frac{S^\ast(n)}{x^\ast(n)}\cdot \frac{x^\ast(n)}{H^\ast(n)}\cdot\frac{H^\ast(n)}{H_-^\ast(n)}
\to 0\cdot 1 \cdot \frac{1}{\lambda}=0.
\]
Thus by \eqref{eq.xminuslower}  we have 
$\liminf_{n\to\infty} x^\ast_-(n)/H_-^\ast(n) \geq 1$, which also implies that $x^\ast_-(n)\to\infty$ as $n\to\infty$. Therefore there exists $N_1\geq 1$ such that 
$x^\ast_-(n)>-x(0)$ for all $n\geq N_1$. Hence by \eqref{eq.xminusupper}, we have for $n\geq N_1$  
\begin{align*}
x^\ast_-(n)&=\max( -x(0), \max_{1\leq l \leq n} (-x(l))) =\max_{1\leq l \leq n} (-x(l)) 
\leq H_-^\ast(n) + S^\ast(n).
\end{align*}
Since $S^\ast(n)/H_-^\ast(n)\to 0$ as $n\to\infty$ we get 
$\limsup_{n\to\infty} x^\ast_-(n)/H_-^\ast(n)\leq 1$. Therefore we have $\lim_{n\to\infty} x^\ast_-(n)/H_-^\ast(n)= 1$. This proves part (i) in the case when $\lambda\in (0,1)$. Part (iii) of Theorem~\ref{theorem.signedfluc} yields the result if $\lambda=1$. 

Next we consider the case where $\lambda\in (1,\infty)$. 
Then $H^\ast\sim H^\ast_-\sim \lambda H^\ast_+$, and $H^\ast_+(n)\to\infty$ as $n\to\infty$. 
By part (ii) of Theorem~\ref{theorem.signedfluc} we have $x_-^\ast(n)/H_-^\ast(n)\to 1$ as $n\to\infty$. 
Furthermore, as $x^\ast(n)/H^\ast(n)\to 1$ as $n\to\infty$ and $S^\ast(n)/x^\ast(n)\to 0$ as $n\to\infty$, we have as $n\to\infty$ that 
\[
\frac{S^\ast(n)}{H_+^\ast(n)}
=
\frac{S^\ast(n)}{x^\ast(n)}\cdot \frac{x^\ast(n)}{H^\ast(n)}\cdot\frac{H^\ast(n)}{H_+^\ast(n)}
\to 0\cdot 1 \cdot \frac{1}{\lambda}=0.
\]
Thus by \eqref{eq.xpluslower}  we have $\liminf_{n\to\infty} x^\ast_+(n)/H_+^\ast(n) \geq 1$, which also implies that $x^\ast_+(n)\to\infty$ as $n\to\infty$. Therefore there exists $N_1\geq 1$ such that 
$x^\ast_+(n)>x(0)$ for all $n\geq N_1$. Hence by \eqref{eq.xplusupper}, we have for $n\geq N_1$  
\begin{align*}
x^\ast_+(n)&=\max_{0\leq j\leq n} x(j) =\max( x(0), \max_{1\leq l \leq n} x(l)) 
=\max_{1\leq l \leq n} x(l) \leq H_+^\ast(n) + S^\ast(n).
\end{align*}
Since $S^\ast(n)/H_+^\ast(n)\to 0$ as $n\to\infty$ we get $\limsup_{n\to\infty} x^\ast_+(n)/H_+^\ast(n)\leq 1$. Hence we have $\lim_{n\to\infty} x^\ast_+(n)/H_+^\ast(n)= 1$. This proves part (i) in the case when $\lambda\in (1,\infty)$, and therefore completes the proof of part (i). 

To prove part (ii), when $\lambda=0$, note part (i) of Theorem~\ref{theorem.signedfluc} already gives $x_+^\ast(n)/H^\ast_+(n)\to 1$ as 
$n\to\infty$. Since $H^\ast\sim H^\ast_+$ as $n\to\infty$ and $S^\ast(n)/x^\ast(n)\to 0$ as $n\to\infty$, we have as $n\to\infty$ 
\[
\frac{S^\ast(n)}{H^\ast_+(n)}
=
\frac{S^\ast(n)}{x^\ast(n)}\cdot\frac{x^\ast(n)}{H^\ast(n)}\cdot\frac{H^\ast(n)}{H^\ast_+(n)}\to 0\cdot 1 \cdot 1=0.
\] 
From \eqref{eq.xminusupper} we have $-x(1)\leq \max_{1\leq l \leq n} (-x(l)) \leq H_-^\ast(n) + S^\ast(n)$ for $n\geq 1$. 
Therefore as $H^\ast_+(n)\to \infty$, $H^\ast_-(n)/H^\ast_+(n)\to 0$  and $S^\ast(n)/H^\ast_+(n)\to 0$ as $n\to\infty$, we have
\[
0\leq \liminf_{n\to\infty} \frac{\max_{1\leq l \leq n} (-x(l))}{H_+^\ast(n)} 
\leq \limsup_{n\to\infty}\frac{\max_{1\leq l \leq n} (-x(l))}{H_+^\ast(n)} \leq 0,
\]
so $\lim_{n\to\infty} \max_{1\leq l \leq n} (-x(l))/H_+^\ast(n)=0$. Since 
\[
x_-^\ast(n)=\max(-x(0),\max_{1\leq l \leq n} (-x(l))),
\] 
and $H^\ast_+(n)\to \infty$ as $n\to\infty$, the above limit yields 
$x_-^\ast(n)/H_+^\ast(n)\to 0$ as $n\to\infty$, as needed.

To prove part (iii), when $\lambda=\infty$, note that part (ii) of Theorem~\ref{theorem.signedfluc} yields  
$x_-^\ast(n)/H^\ast_-(n)\to 1$ as 
$n\to\infty$. Since $H^\ast(n)\sim H^\ast_-(n)$ as $n\to\infty$ and $S^\ast(n)/x^\ast(n)\to 0$ as $n\to\infty$, we have as $n\to\infty$ 
\[
\frac{S^\ast(n)}{H^\ast_-(n)}
=
\frac{S^\ast(n)}{x^\ast(n)}\cdot\frac{x^\ast(n)}{H^\ast(n)}\cdot\frac{H^\ast(n)}{H^\ast_-(n)}\to 0\cdot 1 \cdot 1=0.
\] 
From \eqref{eq.xplusupper} we have
\begin{equation*} 
x(1)\leq \max_{1\leq l \leq n} x(l) \leq H_+^\ast(n) + S^\ast(n), \quad n\geq 1.
\end{equation*}
Therefore as $H^\ast_-(n)\to \infty$, $H^\ast_+(n)/H^\ast_-(n)\to 0$  and $S^\ast(n)/H^\ast_-(n)\to 0$ as $n\to\infty$, we have
\[
0\leq \liminf_{n\to\infty} \frac{\max_{1\leq l \leq n} x(l)}{H_-^\ast(n)} 
\leq \limsup_{n\to\infty}\frac{\max_{1\leq l \leq n} x(l)}{H_-^\ast(n)} \leq 0,
\]
so $\lim_{n\to\infty} \max_{1\leq l \leq n} x(l)/H_-^\ast(n)=0$. Since $x_+^\ast(n)=\max(x(0),\max_{1\leq l \leq n} x(l))$, and $H^\ast_-(n)\to \infty$ as $n\to\infty$, the above limit yields 
$x_+^\ast(n)/H_-^\ast(n)\to 0$ as $n\to\infty$, as needed.

\subsection{Proof of Theorem~\ref{theorem.signedfluct}}
By \eqref{eq.fphi}, for every $\epsilon>0$ there is an $x_1(\epsilon)>0$ such that $|f(x)|\leq (1+\epsilon)\phi(|x|)$ for $|x|\geq x_1(\epsilon)$. By continuity of $f$, we have that $|f(x)|\leq \max_{|x|\leq x_1(\epsilon)}|f(x)|=:F_2(\epsilon)$. Hence we have the bound 
$|f(x)|\leq F_2(\epsilon)+(1+\epsilon)\phi(|x|)$ for all $x\in \mathbb{R}$. Therefore, with $S(n)$ defined as in \eqref{eq.Sn}, we obtain the bound
\[
|S(n)|\leq \sum_{j=0}^n |k(n-j)|\left\{F_2(\epsilon)+(1+\epsilon)\phi(|x(j)|)\right\}.
\]
Since $|x(j)|\leq x^\ast(n)$ for $j\in \{0,\ldots,n\}$ and $\phi$ is increasing (by \eqref{eq.fphi}), we have 
\[
|S(n)|\leq |k|_1 F_2(\epsilon) + (1+\epsilon)|k|_1 \phi(x^\ast(n)). 
\]
The monotonicity of $\phi$ and $x^\ast$ therefore imply that $S^\ast(n)\leq |k|_1 F_2(\epsilon) + (1+\epsilon)|k|_1 \phi(x^\ast(n))$, 
and therefore that 
\begin{equation} \label{eq.sastphibound}
\limsup_{n\to\infty} \frac{S^\ast(n)}{\phi(x^\ast(n)))}\leq |k|_1. 
\end{equation}
Since $\lambda=0$, we have that $x^\ast_+(n)\sim x^\ast(n)\sim H^\ast_+(n)$ as $n\to\infty$. We now show that 
\begin{equation} \label{eq.phipreserves}
\text{$a(n)\sim b(n)$ as $n\to\infty$ and $a(n)\to\infty$ implies $\phi(a(n))\sim \phi(b(n))$ as $n\to\infty$.}
\end{equation}
\eqref{eq.phipreserves} implies that $\phi(x^\ast(n))\sim \phi(H^\ast_+(n))$ and as $\phi$ is asymptotic to $f$, this implies 
from \eqref{eq.sastphibound} that 
\begin{equation} \label{eq.SastfHast}
\limsup_{n\to\infty} \frac{S^\ast(n)}{f(H^\ast_+(n)))}\leq |k|_1. 
\end{equation}
We prove next \eqref{eq.phipreserves}. By hypothesis, for every $\epsilon\in(0,1)$ we have that 
$(1-\epsilon)a(n)<b(n)<(1+\epsilon)a(n)$ for $n\geq N_1(\epsilon)$. Since $\phi$ is increasing we have for $n\geq N_1(\epsilon)$ that  
\[
 \frac{\phi((1-\epsilon)a(n))}{\phi(a(n))}<\frac{\phi(b(n))}{\phi(a(n))}<\frac{\phi((1+\epsilon)a(n))}{\phi(a(n))}.
\]
Next, as $x\mapsto \phi(x)/x$ is decreasing, we have that $\phi((1+\epsilon)x)/((1+\epsilon)x)\leq \phi(x)/x$ for $x>x_0$, or 
 $\phi((1+\epsilon)x)/\phi(x)\leq 1+\epsilon$ for $x>x_0$. Similarly, 
$\phi((1-\epsilon)x)/((1-\epsilon)x)\geq \phi(x)/x$ for $x>x_0/(1-\epsilon)$, or  $\phi((1-\epsilon)x)/\phi(x)\geq 1-\epsilon$. 
Now, since $a(n)\to\infty$ as $n\to\infty$ we have $a(n)>x_0/(1-\epsilon)$ for all $n\geq N_2(\epsilon)$. Let $N_3=\max(N_1,N_2)$. 
Then for $n\geq N_3(\epsilon)$ we have 
\[
1-\epsilon\leq 
 \frac{\phi((1-\epsilon)a(n))}{\phi(a(n))}<\frac{\phi(b(n))}{\phi(a(n))}<\frac{\phi((1+\epsilon)a(n))}{\phi(a(n))}
\leq 1+\epsilon.
\]
Since $\epsilon$ is arbitrary, we have $\phi(b(n))/\phi(a(n))\to 1$ as $n\to\infty$. This proves \eqref{eq.phipreserves}.

We prove now part (ii). By hypothesis we have that $H^\ast_-(n)/f(H^\ast_+(n))\to \infty$ as $n\to\infty$. Then by 
 \eqref{eq.SastfHast} we have 
\[
0\leq 
\limsup_{n\to\infty} \frac{S^\ast(n)}{H^\ast_-(n)}  
= \limsup_{n\to\infty} \frac{S^\ast(n)}{f(H^\ast_+(n)))}\cdot \frac{f(H^\ast_+(n))}{H^\ast_-(n)}=0.  
\]
We have already shown that 
\begin{equation} \label{eq.xminest}
H^\ast_-(n)\leq x_-^\ast(n)+S^\ast(n),\quad \max_{1\leq j\leq n} -x(j)\leq  H_-^\ast(n)+S^\ast(n), \quad n\geq 1.
\end{equation}
From the first inequality in \eqref{eq.xminest}, we have $\liminf_{n\to\infty} x^\ast_-(n)/H^\ast_-(n)\geq 1$, and so $x^\ast_-(n)\to\infty$ as $n\to\infty$. 
Therefore, there is $N_4>1$ such that $x^\ast_-(n)=\max_{1\leq j\leq n} -x(j)$ for $n\geq N_4$. Hence, by \eqref{eq.xminest} we have 
for $n\geq N_4$ the inequality $x_-^\ast(n)\leq  H_-^\ast(n)+S^\ast(n)$, and so $\limsup_{n\to\infty} x^\ast_-(n)/H^\ast_-(n)\leq 1$. 
Hence we have $x^\ast_-(n)/H^\ast_-(n)\to 1$ as $n\to\infty$, as claimed in part (ii). 

For part (iii), we have that $H_-^\ast(n)/f(H_+^\ast(n))\to \lambda_2\in (0,\infty)$ as $n\to\infty$. Hence, by hypothesis, 
\eqref{eq.SastfHast} and from the second inequality in 
\eqref{eq.xminest} we have 
\begin{align*}
\limsup_{n\to\infty}
\frac{x^\ast_-(n)}{f(H^\ast_+(n))}
&\leq
\limsup_{n\to\infty}
\frac{\max_{1\leq j\leq n} -x(j)}{f(H^\ast_+(n))}\\
&\leq  
\limsup_{n\to\infty} \frac{H_-^\ast(n)}{f(H^\ast_+(n))}+ \limsup_{n\to\infty}\frac{S^\ast(n)}{f(H^\ast_+(n))}
\leq \lambda_2+\sum_{j=0}^{\infty} |k(j)|.
\end{align*}
From the definition of $\lambda_2$ the claim follows. If $\lambda_2>\sum_{j=0}^\infty |k(j)|$, by hypothesis, \eqref{eq.SastfHast} and from the second inequality in 
\eqref{eq.xminest} we have 
\begin{align*}
\liminf_{n\to\infty}
\frac{\max_{1\leq j\leq n} -x(j)}{f(H^\ast_+(n))}
&\geq  
\liminf_{n\to\infty} \left\{\frac{H_-^\ast(n)}{f(H^\ast_+(n))}- \frac{S^\ast(n)}{f(H^\ast_+(n))}\right\}\\
&
=\lambda_2+\liminf_{n\to\infty} - \frac{S^\ast(n)}{f(H^\ast_+(n))}\geq \lambda_2-\sum_{j=0}^{\infty} |k(j)|>0.
\end{align*}
Therefore $\max_{1\leq j\leq n} -x(j)\to\infty$ as $n\to\infty$, and so $x_-(n)=\max_{1\leq j\leq n} -x(j)$ for all $n$ sufficiently 
large. The above inequality, this fact, and the definition of $\lambda_2\in (0,\infty)$ proves the last claim in (iii). The proof of part (iv) follows as the proof of the first part of (iii) above.  

\section{Proof of Theorem~\ref{theorem.limsup}}
\subsection{Preliminary result}
We start by proving a preliminary result.
\begin{lemma} \label{H(n)/a(n)}
If $\big(a(n)\big)_{n\geq 1}$ is an increasing sequence with $a(n)\to\infty$ as $n\to\infty$, then
\begin{itemize}
\item[(a)]
\begin{equation*}
\limsup_{n\to\infty}\frac{\left|H(n)\right|}{a(n)}=1 \text{ is equivalent to }
\limsup_{n\to\infty}\frac{H^\ast(n)}{a(n)}=1.
\end{equation*}
\item[(b)]
\[
\limsup_{n\to\infty} \frac{|H(n)|}{a(n)}=+\infty \text{ is equivalent to }
\limsup_{n\to\infty} \frac{H^\ast(n)}{a(n)}=+\infty.
\]
\item[(c)] 
\[
\limsup_{n\to\infty} \frac{|H(n)|}{a(n)}=0 \text{ is equivalent to }
\limsup_{n\to\infty} \frac{H^\ast(n)}{a(n)}=0.
\]
\end{itemize}
\end{lemma}
\begin{proof}
We start with the proof of part (a), and begin by proving that the first statement implies the second. By hypothesis $\left|H(n)\right|<(1+\epsilon)a(n)$ for all $n\geq N_1(\epsilon)$.
Let $n\geq N_1+1$. Then
\begin{align*}
\max_{1\leq j\leq n }\left|H(j)\right|&=\max\bigg(\max_{ 1 \leq j \leq N_1}\left|H(j)\right|,\max_{N_1+1 \leq j \leq n}\left|H(j)\right|\bigg)\\
&\leq \max\bigg(\max_{1\leq j \leq N_1 }\left|H(j)\right|,\max_{ N_1+1 \leq j \leq n }(1+\epsilon)a(j)\bigg)\\
&=\max\bigg(\max_{1\leq j \leq N_1 }\left|H(j)\right|,(1+\epsilon)a(n)\bigg).
\end{align*}
Thus $\limsup_{n\to\infty}\max_{1 \leq j \leq n } |H(j)|/a(n)\leq 
1+\epsilon$. Letting $\epsilon\to 0$ gives
\begin{equation} \label{eq.lem1}\limsup_{n\to\infty}\frac{\max_{1\leq j \leq n}\left|H(j)\right|}{a(n)}\leq 1. \end{equation}
Since $\max_{1\leq j\leq n }\left|H(j)\right|\geq \left|H(n)\right|$, we have that 
\begin{equation} \label{eq.lem2} 
\limsup_{n\to\infty}\frac{\max_{1 \leq j \leq n}\left|H(j)\right|}{a(n)}\geq \limsup_{n\to\infty}\frac{\left|H(n)\right|}{a(n)}=1
\end{equation}
by hypothesis. Combining \eqref{eq.lem1} and \eqref{eq.lem2} proves the first implication in part (a). 

To prove the reverse implication in part (a), we begin with the assumption that
\begin{equation} \label{eq.lem3}
\limsup_{n\to\infty}\frac{\max_{1\leq j\leq n }\left|H(j)\right|}{a(n)}=1. \end{equation}
Again, we have that
$\max_{1\leq j \leq n }\left|H(j)\right|\geq \left|H(n)\right|$ and so
\begin{equation*}
\limsup_{n\to\infty}\frac{\left|H(n)\right|}{a(n)}\leq \limsup_{n\to\infty}\frac{H^*(n)}{a(n)}=1. 
\end{equation*}
Therefore, there exists $\lambda\in [0,1]$ such that $\limsup_{n\to\infty} |H(n)|/a(n)=\lambda$.
Now suppose $\lambda\in (0,1)$ and define $a_\lambda(n):=\lambda a(n)$. Then
$\limsup_{n\to\infty}\left|H(n)\right|/a_\lambda(n)=1$. Applying the implication from the first part of the lemma and noting that $a_\lambda(n)$ is increasing with  $a_\lambda(n)\to\infty$ as $n\to\infty$ gives
$\limsup_{n\to\infty} H^\ast(n)/a_\lambda(n)=1$ so 
$\limsup_{n\to\infty} H^\ast(n)/a(n)=\lambda$. However, since we have assumed $\lambda\in (0,1)$, this contradicts \eqref{eq.lem3} and hence we can have $\lambda=0$ or $\lambda=1$. Ruling out the case $\lambda=0$ proves the statement.

If $\lambda=0$, then for every $\epsilon\in(0,1)$ there is $N_1(\epsilon)$ such that $\left|H(n)\right|<\epsilon a(n)$ for all $n\geq N_1(\epsilon)$. Following the same argument as in the first part of the proof, this gives 
\begin{equation*}
 \max_{1\leq j \leq n}\left|H(j)\right|\leq\max\left(\epsilon a(n),\max_{1 \leq j \leq N_1 }\left|H(j)\right|\right),\quad n\geq N_1+1. 
\end{equation*}
Thus 
letting $n\to\infty$ now yields $\limsup_{n\to\infty} H^\ast(n)/a(n)\leq \epsilon$ and letting $\epsilon\to0^+$ gives
$\limsup_{n\to\infty} H^\ast(n)/a(n)=0$. This again results in a contradiction, since by assumption 
$\limsup_{n\to\infty} H^\ast(n)/a(n)=1$. 
Hence $\lambda=1$, concluding the proof of part (a).

To prove the forward implication in (b), note that $|H(n)|\leq \max_{1\leq j\leq n}|H(j)|$. To prove the reverse implication, suppose otherwise, i.e., that
\[
\lambda:=\limsup_{n\to\infty} |H(n)|/a(n) \in [0,\infty).
\] 
Then there is $N>0$ such that for all $n\geq N$,  $|H(n)|\leq (\lambda+1)a(n)$. Hence for $n\geq N$ we may
use the monotonicity of $a$ to get 
\begin{align*}
H^\ast(n)&=\max\left(\max_{1\leq j\leq N-1} |H(j)|,\max_{N\leq j \leq n} |H(j)|\right) 
\leq \max_{1\leq j\leq N-1} |H(j)| + \max_{N\leq j \leq n} |H(j)| \\
&=\max_{1\leq j\leq N-1} |H(j)| + (\lambda+1) a(n).
\end{align*}
Since $a(n)\to\infty$ as $n\to\infty$, we have 
$+\infty=\limsup_{n\to\infty} H^\ast(n)/a(n)\leq \lambda+1<+\infty$, a contradiction. 

To prove part (c), note that because $|H(n)|\leq H^\ast(n)$, the second statement in (c) proves the first. Suppose the first statement is true. Then for every $\epsilon>0$ there is an $N(\epsilon)>0$ such that $|H(n)|<\epsilon a(n)$ 
for all $n\geq N(\epsilon)$. Thus for $n\geq N(\epsilon)$ we may use the monotonicity of $a$ to get 
\begin{align*}
H^\ast(n)&=\max\left(\max_{1\leq j\leq N(\epsilon)-1} |H(j)|,\max_{N(\epsilon)\leq j \leq n} |H(j)|\right)\\
&\leq \max_{1\leq j\leq N(\epsilon)-1} |H(j)| + \max_{N(\epsilon)\leq j \leq n} |H(j)| \\
&=\max_{1\leq j\leq N(\epsilon)-1} |H(j)| + \epsilon a(n).
\end{align*}
Since $a(n)\to\infty$ as $n\to\infty$, we have 
$\limsup_{n\to\infty} H^\ast(n)/a(n)\leq \epsilon$, and letting $\epsilon\to 0$ completes the proof.
\end{proof}

\subsection{Proof of Theorem~\ref{theorem.limsup}}
We start with the proof that statement (a) implies statement (b) implies statement (c) implies statement (d) implies statement (a) in the case that $\rho\in (0,\infty)$. Define $a_\rho(n)=\rho a(n)$. Then statement (a) implies
$\limsup_{n\to\infty} |H(n)|/a_\rho(n)=1$. 
By part (a) of Lemma~\ref{H(n)/a(n)}, it follows that $\limsup_{n\to\infty} H^\ast(n)/a_\rho(n)=1$. 
This implies statement (b). Since $x^\ast(n)/H^\ast(n)\to 1$ as $n\to\infty$, we have 
$\limsup_{n\to\infty} x^\ast(n)/a_\rho(n)=1$. This is statement (d). From part (a) of Lemma~\ref{H(n)/a(n)} it follows that this is equivalent to $\limsup_{n\to\infty} \left|x(n)\right|/a_\rho(n)=1$ which is equivalent to statement (c). 
Since $x^\ast(n)/H^\ast(n)\to 1$ as $n\to\infty$, from statement (d) we have 
$\limsup_{n\to\infty} H^\ast(n)/a_\rho(n)=1$.
By Lemma~\ref{H(n)/a(n)} this is equivalent to 
$\limsup_{n\to\infty} |H(n)|/a_\rho(n)=1$, which is precisely statement (a).

Suppose now $\rho=0$. Again we show that (a) implies (b) implies (c) implies (d) implies (a). By part (c) of Lemma~\ref{H(n)/a(n)}, 
(a) implies $\max_{1\leq j \leq n}\left|H(j)\right|/a(n)\to 0$ as $n\to\infty$, which is part (b). Since $x^\ast(n)/H^\ast(n)\to 1$ 
as $n\to\infty$. Hence 
$x^\ast(n)/a(n)\to 0$ as $n\to\infty$, which is statement (d). By part (c) of Lemma~\ref{H(n)/a(n)}, this is equivalent to (c). From statement (d) and $x^\ast(n)/H^\ast(n)\to 1$ as $n\to\infty$, we have $\limsup_{n\to\infty} H^\ast(n)/a(n)=0$. This is equivalent to statement (a) by  part (c) of Lemma~\ref{H(n)/a(n)}.

Suppose lastly that $\rho=\infty$. Statement (a) implies statement (b) from part (b) of Lemma~\ref{H(n)/a(n)}. 
Since $x^\ast(n)/H^\ast(n)\to 1$ as $n\to\infty$, we have 
$\limsup_{n\to\infty} x^\ast(n)/a(n)=+\infty$. By Lemma~\ref{H(n)/a(n)} part (b), we have $\limsup_{n\to\infty} |x(n)|/a(n)=+\infty$, 
which is (c). Applying Lemma~\ref{H(n)/a(n)} part (b) again gives statement (d). Finally, if (d) holds, since $x^\ast(n)/H^\ast(n)\to 1$ 
as $n\to\infty$ 
we have that 
\[
\limsup_{n\to\infty} \frac{H^\ast(n)}{a(n)}=\limsup_{n\to\infty} \frac{H^\ast(n)}{x^\ast(n)}\cdot \frac{x^\ast(n)}{a(n)}=+\infty.
\]
By part (b) of Lemma~\ref{H(n)/a(n)}, this is equivalent to statement (a), as required.

\subsection{Proof of Theorem~\ref{theorem.limsupsignedfluc2}}
The proof of Theorem~\ref{theorem.limsupsignedfluc2} requires a result almost parallel to Lemma~\ref{H(n)/a(n)}. 
We start by proving this auxiliary result.
\begin{lemma} \label{H(n)/a(n)2}
If $\big(a(n)\big)_{n\geq 1}$ is an increasing sequence with $a(n)\to\infty$ as $n\to\infty$, then
\begin{itemize}
\item[(a)]
\begin{equation*}
\limsup_{n\to\infty}\frac{H(n)}{a(n)}=1 \text{ is equivalent to }
\limsup_{n\to\infty}\frac{\max_{ 1 \leq j \leq n } H(j)}{a(n)}=1.
\end{equation*}
\item[(b)] 
\[
\limsup_{n\to\infty} \frac{H(n)}{a(n)}=0 \text{ is equivalent to }
\limsup_{n\to\infty} \frac{\max_{1\leq j\leq n} H(j)}{a(n)}=0.
\]
\end{itemize}
\end{lemma}
\begin{proof}
We start with the proof of part (a), and prove the left to right implication first. By hypothesis $H(n)<(1+\epsilon)a(n)$ for all $n\geq N_1(\epsilon)$.
Let $n\geq N_1+1$. Then
\begin{align*}
\max_{1\leq j \leq n} H(j)
&=\max\bigg(\max_{1\leq j\leq N_1} H(j),\max_{N_1+1\leq j\leq n} H(j)\bigg)\\
&\leq \max\bigg(\max_{1\leq j\leq N_1}H(j),\max_{N_1+1\leq j \leq n}(1+\epsilon)a(j)\bigg)\\
&=\max\bigg(\max_{1\leq j\leq N_1} H(j),(1+\epsilon)a(n)\bigg).
\end{align*}
Thus $\limsup_{n\to\infty}\max_{ 1\leq j\leq n} H(j)/a(n)\leq 1+\epsilon$, and letting $\epsilon\to 0$ gives
\begin{equation} \label{eq.lem1+}
\limsup_{n\to\infty}\frac{\max_{1\leq j \leq n} H(j)}{a(n)}\leq 1. 
\end{equation}
Since $\max_{1\leq j\leq n} H(j)\geq H(n)$, we have that 
\begin{equation} \label{eq.lem2+} 
\limsup_{n\to\infty}\frac{\max_{1\leq j\leq n} H(j)}{a(n)}\geq \limsup_{n\to\infty}\frac{H(n)}{a(n)}=1
\end{equation}
by hypothesis. Combining \eqref{eq.lem1+} and \eqref{eq.lem2+} proves the forward implication. 

To prove the reverse implication in part (a), we begin with the assumption that
\begin{equation} \label{eq.lem3+}
\limsup_{n\to\infty}\frac{\max_{1 \leq j \leq n} H(j)}{a(n)}=1. 
\end{equation}
Again, we have that $\max_{1\leq j \leq n} H(j)\geq H(n)$ and so
\[
\lambda:=
\limsup_{n\to\infty}\frac{H(n)}{a(n)}\leq \limsup_{n\to\infty}\frac{H^*_+(n)}{a(n)}=1. 
\]
Also, by \eqref{eq.lem3+}, there is an integer sequence $\tau_n\uparrow\infty$ such that $\max_{1 \leq j \leq \tau_n} H(j)>(1-\epsilon)a(\tau_n)\to\infty$ as $n\to\infty$. Thus $\max_{1 \leq j \leq n} H(j)\to\infty$, and so $\limsup_{n\to\infty} H(n)=+\infty$. 
Hence $\lambda\geq 0$. Therefore, $\lambda\in [0,1]$. Now suppose $\lambda\in (0,1)$ and define $a_\lambda(n):=\lambda a(n)$. Then we have 
$\limsup_{n\to\infty} H(n)/a_\lambda(n)=1$.
Applying the implication from the first part of the lemma and noting that $a_\lambda(n)$ is increasing with  $a_\lambda(n)\to\infty$ as $n\to\infty$, we arrive at the limit 
$\limsup_{n\to\infty} \max_{1\leq j\leq n} H(j)/a_\lambda(n)=1$. Therefore we have 
$\limsup_{n\to\infty} \max_{1\leq j \leq n} H(j)/a(n)=\lambda$. 
However, since we have assumed $\lambda\in (0,1)$, this contradicts \eqref{eq.lem3+} and hence we can have $\lambda=0$ or $\lambda=1$. Ruling out the case $\lambda=0$ proves the statement.

If $\lambda=0$, then for every $\epsilon\in(0,1)$ there is $N_1(\epsilon)$ such that $H(n)<\epsilon a(n)$ for all 
$n\geq N_1(\epsilon)$. Following the same argument as in the first part of the proof, this gives 
\begin{equation*}
 \max_{1\leq j\leq n} H(j)\leq\max\left(\epsilon a(n),\max_{1\leq j\leq N_1} H(j)\right),\qquad n\geq N_1+1. 
\end{equation*}
Thus letting $n\to\infty$ now yields
$\limsup_{n\to\infty} \max_{1\leq j\leq n} H(j)/a(n)\leq \epsilon$ and letting $\epsilon\to0^+$ gives
$\limsup_{n\to\infty} \max_{1\leq j\leq n} H(j)/a(n)\leq 0$. 
On the other hand, $\max_{1\leq j\leq n} H(j)\geq H(1)$ so 
$\limsup_{n\to\infty} \max_{1\leq j\leq n} H(j)/a(n)\geq 0$. Therefore, we have  
 $\limsup_{n\to\infty} \max_{1\leq j\leq n} H(j)/a(n) =0$. This is again a contradiction since by assumption 
$\limsup_{n\to\infty} \max_{1\leq j \leq n} H(j)/a(n)=1$. Therefore $\lambda=1$ which concludes the proof of part (a).

To prove part (b), note that because $H(n)\leq H_+^\ast(n)$, the second statement in (b) proves the first. Suppose the first is true. Then $H(n)<\epsilon a(n)$ 
for all $n\geq N(\epsilon)$. Thus for $n\geq N(\epsilon)$ we may use the monotonicity of $a$ to get 
\begin{align*}
H^\ast_+(n)&=\max\left(\max_{1\leq j\leq N(\epsilon)-1} H(j),\max_{N(\epsilon)\leq j \leq n} H(j)\right) 
\leq \max\left(\max_{1\leq j\leq N(\epsilon)-1} H(j),\epsilon a(n)\right) 
\end{align*}
Since $a(n)\to\infty$ as $n\to\infty$, we have $\limsup_{n\to\infty} H^\ast_+(n)/a(n)\leq \epsilon$, and letting $\epsilon\to 0$ completes the proof.
\end{proof}

\subsection{Proof of Theorem~\ref{theorem.limsupsignedfluc2}}
Suppose $\lambda\in (0,\infty)$. Therefore, with $a(n)=a_+(n)$
\[
\limsup_{n\to\infty} \frac{|H(n)|}{a(n)}=\limsup_{n\to\infty} \frac{\max(H(n),-H(n))}{a_+(n)}=\max(\rho_+,\rho_-\lambda).
\]
Therefore by Theorem~\ref{theorem.limsup}
\[
\limsup_{n\to\infty} \frac{|x(n)|}{a(n)}=\max(\rho_+,\rho_-\lambda), 
\quad
\limsup_{n\to\infty} \frac{x^\ast(n)}{a(n)}=\max(\rho_+,\rho_-\lambda).
\]
Moreover, by Lemma~\ref{H(n)/a(n)2}
\[
\limsup_{n\to\infty}\frac{H_+^\ast(n)}{a_+(n)}=:\rho_+\in (0,\infty],
\]
so
\[
\limsup_{n\to\infty}\frac{H_+^\ast(n)}{a(n)}=:\rho_+\in (0,\infty],
\]
As usual, we have for every $\epsilon>0$ that 
$S^\ast(n)\leq |k|_1 F(\epsilon)+\epsilon |k|_1 x^\ast(n)$ for $n\geq 1$. 
Therefore, we have that $S^\ast(n)/a(n)\to 0$ as $n\to\infty$, so $S^\ast(n)/a_+(n)\to 0$ as $n\to\infty$ and as $\lambda\in (0,\infty)$, $S^\ast(n)/a_-(n)\to 0$ as $n\to\infty$.
Since $x_+^\ast(n)\geq H_+^\ast(n)-S^\ast(n)$ and $\max_{1\leq j\leq n} x(j) \leq H_+^\ast(n)+S^\ast(n)$, we have 
\[
\limsup_{n\to\infty} \frac{x_+^\ast(n)}{a_+(n)}\geq \limsup_{n\to\infty} \left\{\frac{H_+^\ast(n)}{a_+(n)}-\frac{S^\ast(n)}{a_+(n)}\right\}
=\limsup_{n\to\infty} \frac{H_+^\ast(n)}{a_+(n)}=\rho_+.
\]
and
\[
\limsup_{n\to\infty} \frac{\max_{1\leq j\leq n} x(j)}{a_+(n)}
\leq \limsup_{n\to\infty} \frac{H_+^\ast(n)}{a_+(n)}+ \limsup_{n\to\infty} \frac{S^\ast(n)}{a_+(n)}=\rho_+.
\]
Since $x_+^\ast(n)=\max(x(0),\max_{1\leq j\leq n} x(j))$, and $\limsup_{n\to\infty} x^\ast_+(n)=\infty$ (so $x^\ast_+(n)\to\infty$ as $n\to\infty$ by monotonicity), we have that $x_+^\ast(n)=\max_{1\leq j\leq n} x(j))$ for all $n$ sufficiently large. Hence
$\limsup_{n\to\infty} \max_{1\leq j\leq n} x(j)/a_+(n)=\rho_+$.
By Lemma~\ref{H(n)/a(n)2}, this gives 
$\limsup_{n\to\infty} x(n)/a_+(n)=\rho_+$ as required.

On the other hand, since 
$x_-^\ast(n)\geq H_-^\ast(n)-S^\ast(n)$ and $\max_{1\leq j\leq n} (-x(j)) \leq H_-^\ast(n)+S^\ast(n)$,
\[
\limsup_{n\to\infty} \frac{x_-^\ast(n)}{a_-(n)}\geq \limsup_{n\to\infty} 
\left\{\frac{H_-^\ast(n)}{a_-(n)}-\frac{S^\ast(n)}{a_-(n)}\right\}
=\limsup_{n\to\infty} \frac{H_-^\ast(n)}{a_-(n)}=\rho_-.
\]
and
\[
\limsup_{n\to\infty} \frac{\max_{1\leq j\leq n} (-x(j))}{a_-(n)}
\leq \limsup_{n\to\infty} \frac{H_-^\ast(n)}{a_-(n)}+ \limsup_{n\to\infty} \frac{S^\ast(n)}{a_-(n)}=\rho_-.
\]
Since $x_-^\ast(n)=\max(-x(0),\max_{1\leq j\leq n} (-x(j)))$, and $\limsup_{n\to\infty} x^\ast_-(n)=\infty$ 
(so $x^\ast_-(n)\to\infty$ as $n\to\infty$ by monotonicity), we have that $x_-^\ast(n)=\max_{1\leq j\leq n} (-x(j)))$ for all $n$ sufficiently large. Hence $\limsup_{n\to\infty} \max_{1\leq j\leq n} -x(j)/a_-(n)=\rho_-$. By Lemma~\ref{H(n)/a(n)2}, this also gives $\limsup_{n\to\infty} -x(n)/a_+(n)=\rho_+$, as required.

Suppose $\lambda=0$. Therefore, with $a(n)=a_+(n)$
\[
\limsup_{n\to\infty} \frac{|H(n)|}{a(n)}=\limsup_{n\to\infty} \frac{\max(H(n),-H(n))}{a_+(n)}=\rho_+.
\]
Thus by Theorem~\ref{theorem.limsup}
$\limsup_{n\to\infty} |x(n)|/a(n)=\rho_+ $ and $\limsup_{n\to\infty} x^\ast(n)/a(n)=\rho_+$.
Moreover, by Lemma~\ref{H(n)/a(n)2}
\[
\limsup_{n\to\infty}\frac{H_+^\ast(n)}{a_+(n)}=:\rho_+\in (0,\infty],
\]
so
\[
\limsup_{n\to\infty}\frac{H_+^\ast(n)}{a(n)}=:\rho_+\in (0,\infty],
\]
Since for every $\epsilon>0$ that $S^\ast(n)\leq |k|_1 F(\epsilon)+\epsilon |k|_1 x^\ast(n)$ for $n\geq 1$, it follows 
that $S^\ast(n)/a(n)\to 0$ as $n\to\infty$, and so $S^\ast(n)/a_+(n)\to 0$ as $n\to\infty$.
Since $x_+^\ast(n)\geq H_+^\ast(n)-S^\ast(n)$ and $\max_{1\leq j\leq n} x(j) \leq H_+^\ast(n)+S^\ast(n)$, we have 
\[
\limsup_{n\to\infty} \frac{x_+^\ast(n)}{a_+(n)}\geq \limsup_{n\to\infty} \left\{\frac{H_+^\ast(n)}{a_+(n)}-\frac{S^\ast(n)}{a_+(n)}\right\}
=\limsup_{n\to\infty} \frac{H_+^\ast(n)}{a_+(n)}=\rho_+.
\]
and
\[
\limsup_{n\to\infty} \frac{\max_{1\leq j\leq n} x(j)}{a_+(n)}
\leq \limsup_{n\to\infty} \frac{H_+^\ast(n)}{a_+(n)}+ \limsup_{n\to\infty} \frac{S^\ast(n)}{a_+(n)}=\rho_+.
\]
Since $x_+^\ast(n)=\max(x(0),\max_{1\leq j\leq n} x(j))$, and $\limsup_{n\to\infty} x^\ast_+(n)=\infty$ (so $x^\ast_+(n)\to\infty$ as $n\to\infty$ by monotonicity), we have that $x_+^\ast(n)=\max_{1\leq j\leq n} x(j)$ for all $n$ sufficiently large. Hence
$\limsup_{n\to\infty} \max_{1\leq j\leq n} x(j)/a_+(n)=\rho_+$. By Lemma~\ref{H(n)/a(n)2}, this gives 
$\limsup_{n\to\infty} x(n)/a_+(n)=\rho_+$, as required.

On the other hand, because $H^\ast_-(n)/a_+(n)\to 0$ as $n\to\infty$, and 
$x_-^\ast(n)\geq H_-^\ast(n)-S^\ast(n)$ and $\max_{1\leq j\leq n} (-x(j)) \leq H_-^\ast(n)+S^\ast(n)$, we have 
\[
\limsup_{n\to\infty} \frac{\max_{1\leq j\leq n} (-x(j))}{a_+(n)}
\leq \limsup_{n\to\infty} \frac{H_-^\ast(n)}{a_+(n)}+ \limsup_{n\to\infty} \frac{S^\ast(n)}{a_+(n)}=0.
\]
On the other hand for $n\geq 1$, we have $\max_{1\leq j\leq n} (-x(j))\geq -x(1)$, so 
\[
\limsup_{n\to\infty} \frac{\max_{1\leq j\leq n} (-x(j))}{a_+(n)}\geq 0.
\] 
Hence $\lim_{n\to\infty} \max_{1\leq j\leq n} (-x(j))/a_+(n)=0$. Since 
\[
x_-^\ast(n)=\max(-x(0),\max_{1\leq j\leq n} (-x(j))),
\] 
we either have that $x^\ast_-(n)$ tends to a finite limit and $\limsup_{n\to\infty} x_-^\ast(n)/a_-(n)=0$,
or $x^\ast_-(n)\to\infty$ as $n\to\infty$, in which case 
$x_-^\ast(n)=\max_{1\leq j\leq n} (-x(j)))$ for all $n$ sufficiently large, and once again we have 
$\limsup_{n\to\infty} x_-^\ast(n)/a_-(n)=0$, as we need. By Lemma~\ref{H(n)/a(n)2}, this also gives 
$\limsup_{n\to\infty} -x(n)/a_+(n)=0$, as required.

\section{Proof of Theorem~\ref{theorem.ergodic}}
We have that $\left|f(x)\right|\leq F(\epsilon)+\epsilon|x|$ for all $x\in\mathbb{R}$. Choose $\epsilon\in(0,1)$ such that $2\epsilon\left|k\right|_1<1$ and $\epsilon=\epsilon(\eta)$ is given by $2\epsilon\left|k\right|_1=1-\frac{1}{1+\eta}$, or equivalently
\begin{equation} \label{eq.eta}
1+\eta=\frac{1}{1-2\epsilon\left|k\right|_1}.
\end{equation}
Since $\sum_{j=0}^{\infty}\left|k(j)\right|>0$ there exists $N\in \mathbb{N}$ such that
\begin{equation} \label{eq.sumkj}
\sum_{j=0}^{n}\left|k(j)\right|\geq \frac{1}{2}\sum_{j=0}^{\infty}\left|k(j)\right|,\quad n\geq N_1. 
\end{equation}
For fixed $n\geq N_1$, define $k_n:[0,...,n)\in [0,\infty)$ by 
\begin{equation} \label{eq.knj}
k_n(j)=\frac{\left|k(j)\right|}{\sum_{l=0}^{n}\left|k(l)\right|}, \quad j=0,...,n. \end{equation}
Then for all $j\in{0,...n}$, $\sum_{j=0}^{n}k_n(j)=1$ and $k_n(j)\geq 0$. 
Hence for any sequence $x$ we have
\begin{align*} \sum_{j=0}^{n}\left|k(j)\right|\left|x(n-j)\right|
&=\sum_{j=0}^{n}k_n(j) \cdot \left|x(n-j)\right| \cdot \sum_{l=0}^{n}\left|k(l)\right|\\
&\leq  |k|_1 \sum_{j=0}^{n}k_n(j) \left|x(n-j)\right|.
\end{align*}
Let $n\geq N_1$, then using this estimate, \eqref{eq.knj} and \eqref{eq.x} we have
\begin{align} 
\left|x(n+1)\right|
&\leq \left|H(n+1)\right|+F(\epsilon)\left|k\right|_1+\epsilon\left|k\right|_1 \sum_{j=0}^{n} k_n(j)\left|x(n-j)\right|. \label{eq.xn}
\end{align}
Define \begin{equation} \label{eq.t1t2}
t_1=1-2\epsilon\left|k\right|_1, \quad t_2=\epsilon\left|k\right|_1 \end{equation}
which gives $t_1=1/(1+\eta)$. Then $t_1,t_2\in (0,1)$ and $1-t_1-t_2=\epsilon\left|k\right|_1$. Now, because $k_n(j)\geq 0$ and $\sum_{j=0}^{n} k_n(j)=1$, we have, by Jensen's Inequality (Lemma~\ref{lemma.Jensen}) that
\begin{equation} \label{eq.phisum}
\varphi\bigg(\sum_{j=0}^{n}k_n(j)\left|x(n-j)\right|\bigg)\leq \sum_{j=0}^{n}k_n(j)\varphi\left(\left|x(n-j)\right|\right). 
\end{equation}
Since $\varphi$ is an increasing function we have, from \eqref{eq.xn} and \eqref{eq.t1t2}, that
\begin{align*} 
\lefteqn{\varphi\left(\left|x(n+1)\right|\right)}\\
&\leq \varphi\bigg(t_1(1+\eta) \left|H(n+1)\right|+t_2\left(\frac{F(\epsilon)}{\epsilon}\right)+(1-t_1-t_2) \sum_{j=0}^{n}\big(k_n(j)\left|x(n-j)\right|\big)\bigg)\\
&=:a_n,
\end{align*}
where we used \eqref{eq.eta} at the penultimate step. 
By the convexity of $\varphi$, and applying Jensen's Inequality (Lemma~\ref{lemma.Jensen}) twice, we have the following:
\begin{align*}
a_n&=\varphi\bigg(t_1(1+\eta) \left|H(n+1)\right|+t_2\left(\frac{F(\epsilon)}{\epsilon}\right)+(1-t_1-t_2) 
\sum_{j=0}^{n}\big(k_n(j)\left|x(n-j)\right|\big)\bigg)\\
&\leq t_1\varphi\big((1+\eta)\left|H(n+1)\right|\big)+t_2\varphi\left(\frac{F(\epsilon)}{\epsilon}\right) 
\\&\qquad+(1-t_1-t_2)\varphi\left(\sum_{j=0}^{n}k_n(j)\left|x(n-j)\right|\right)\\
&=\frac{1}{1+\eta}\varphi\big((1+\eta)\left|H(n+1)\right|\big) + \epsilon\left|k\right|_1\varphi\left(\frac{F(\epsilon)}{\epsilon}\right)\\
&\qquad+\epsilon\left|k\right|_1 \sum_{j=0}^{n}k_n(j)\varphi\left(\left|x(n-j)\right|\right)
\end{align*}
by \eqref{eq.phisum}. Hence, for $n\geq N_1$
\begin{align*}
\lefteqn{
\varphi\left(\left|x(n+1)\right|\right)}
\\&\leq \frac{1}{1+\eta}\varphi\big((1+\eta)\left|H(n+1)\right|\big) + \epsilon\left|k\right|_1\varphi\left(\frac{F(\epsilon)}{\epsilon}\right)+\epsilon\left|k\right|_1 \sum_{j=0}^{n}k_n(j)\varphi\left(\left|x(n-j)\right|\right)\\
&=\frac{1}{1+\eta}\varphi\big((1+\eta)\left|H(n+1)\right|\big) 
+ \epsilon\left|k\right|_1\varphi\left(\frac{F(\epsilon)}{\epsilon}\right)\\
&\qquad \qquad\qquad\qquad
+\epsilon\left|k\right|_1 \sum_{j=0}^{n} \frac{\left|k(j)\right|}{\sum_{l=0}^{n}\left|k(l)\right|} \varphi\left(\left|x(n-j)\right|\right).
\end{align*}
If $n\geq N_1$, then by \eqref{eq.sumkj} we have $\sum_{j=0}^{n}\left|k(l)\right|\geq \frac{1}{2}\left|k\right|_1$. Thus
\begin{multline*}
\varphi\left(\left|x(n+1)\right|\right)\leq \frac{\varphi\big((1+\eta)\left|H(n+1)\right|\big)}{1+\eta} + \epsilon\left|k\right|_1\varphi\left(\frac{F(\epsilon)}{\epsilon}\right) \\
+2\epsilon \sum_{j=0}^{n} \left|k(j)\right|\varphi\left(\left|x(n-j)\right|\right).
\end{multline*}
Now define $\varphi_x^*:=\max_{{0}\leq{j}\leq{N_1+1}} \varphi\left(\left|x(j)\right|\right)$. 
Then 
\begin{multline*} 
\varphi\left(\left|x(n+1)\right|\right)\leq \varphi_x^*+C(\eta)+\frac{1}{1+\eta} \varphi\big((1+\eta)\left|H(n+1)\right|\big) \\
+2\epsilon(\eta) \sum_{j=0}^{n} \left|k(n-j)\right|\varphi\left(\left|x(j)\right|\right), 
\end{multline*}
where
\begin{equation*} 
C(\eta)=\epsilon(\eta)\left|k\right|_1\varphi\left(\frac{F(\epsilon(\eta))}{\epsilon(\eta)}\right)<+\infty. \end{equation*}
For $0\leq n\leq N_1$,
\begin{equation*} \varphi\left(\left|x(n+1)\right|\right)\leq \max_{{0}\leq{j}\leq{N_1}}\varphi\left(\left|x(j+1)\right|\right) \leq \max_{{0}\leq{j}\leq{N_1}} \varphi\left(\left|x(j)\right|\right) =\varphi_x^*. \end{equation*}
Thus for every $n\geq 0$,
\begin{multline*}
\varphi\left(\left|x(n+1)\right|\right)\leq \varphi_x^*+C(\eta)+ \frac{1}{1+\eta} \varphi\big((1+\eta)\left|H(n+1)\right|\big) 
\\+2\epsilon(\eta) \sum_{j=0}^{n} \left|k(n-j)\right|\varphi\left(\left|x(j)\right|\right),
\end{multline*}
and hence for all $N\geq 0$
\begin{multline*}
\sum_{n=0}^{N} \varphi\left(\left|x(n+1)\right|\right)
\leq \big(\varphi_x^*+C(\eta)\big)(N+1)+\frac{1}{1+\eta}\sum_{n=0}^{N} \varphi\big((1+\eta)\left|H(n+1)\right|
\\+2\epsilon(\eta) \sum_{n=0}^{N}\sum_{j=0}^{n} \left|k(n-j)\right|\varphi\left(\left|x(j)\right|\right).
\end{multline*}
Now, keeping in mind that $0\leq j\leq n\leq N$,
\begin{align*}
\sum_{n=0}^{N}\sum_{j=0}^{n} \left|k(n-j)\right|\varphi\left(\left|x(j)\right|\right)&=\sum_{j=0}^{N}\bigg(\sum_{n=j}^{N} \left|k(n-j)\right|\bigg)\varphi\left(\left|x(j)\right|\right)
\leq \left|k\right|_1 \sum_{j=0}^{N} \varphi\left(\left|x(j)\right|\right).
\end{align*}
Thus for $n\geq 0$ we have
\begin{multline*}
\sum_{j=0}^{n} \varphi\left(\left|x(j+1)\right|\right)\leq \big(\varphi_x^*+C(\eta)\big)(n+1)
+\frac{1}{1+\eta}\sum_{j=0}^{n} \varphi\big((1+\eta)\left|H(j+1)\right|\big)\\
+2\epsilon(\eta)\left|k\right|_1 \sum_{j=0}^{n} \varphi\left(\left|x(j)\right|\right).
\end{multline*}
Therefore, defining $S_n:=\sum_{j=0}^n \varphi\left(\left|x(j)\right|\right)$, we have
\begin{equation*} 
\sum_{j=1}^{n+1} \varphi\left(\left|x(j)\right|\right)
\leq \big(\varphi_x^*+C(\eta)\big)(n+1)
+\frac{1}{1+\eta}\sum_{j=1}^{n+1} \varphi\big((1+\eta)\left|H(j)\right|\big)+2\epsilon(\eta)\left|k\right|_1S_n, \end{equation*}
and so, as $S_n\leq S_{n+1}$, we have
\begin{equation*}
S_{n+1}\leq \varphi\left(\left|x(0)\right|\right)+\big(\varphi_x^*+C(\eta)\big)(n+1)
+\frac{1}{1+\eta}\sum_{j=1}^{n+1}\varphi\big((1+\eta)\left|H(j)\right|\big)+2\epsilon(\eta)\left|k\right|_1S_{n+1}.
\end{equation*}
Since $1-2\epsilon(\eta)\left|k\right|_1>0$, we have, for all $n\geq 0$
\begin{multline*}
 \big(1-2\epsilon(\eta)\left|k\right|_1\big) \sum_{j=0}^{n+1}\varphi\left(\left|x(j)\right|\right)
\leq \varphi\left(\left|x(0)\right|\right) 
+\big(\varphi_x^*+C(\eta)\big)(n+1)\\
+\frac{1}{1+\eta}\sum_{j=1}^{n+1}\varphi\big((1+\eta)\left|H(j)\right|\big).
\end{multline*}
Hence, for $n\geq 1$
\begin{equation*}
\sum_{j=0}^{n}\varphi\left(\left|x(j)\right|\right)\leq (1+\eta)\varphi\left(\left|x(0)\right|\right) 
+\big(\varphi_x^*+C(\eta)\big)n(1+\eta)+\sum_{j=1}^{n} \varphi\big((1+\eta)\left|H(j)\right|\big)
\end{equation*}
and consequently,
\begin{align*} 
\limsup_{n\to\infty}\frac{1}{n} \sum_{j=0}^{n}\varphi\left(\left|x(j)\right|\right)
&\leq \big(\varphi_x^*+C(\eta)\big)(1+\eta)
+ \limsup_{n\to\infty}\frac{1}{n} \sum_{j=1}^{n} \varphi\big((1+\eta)\left|H(j)\right|\big)\\
&<+\infty.
\end{align*}
By assumption, this proves part (i). To prove part (ii), let $\epsilon\in(0,1)$ be such that 
$2\epsilon\left|k\right|_1<1$ and $\epsilon$ obeys \eqref{eq.eta}.
Rearranging \eqref{eq.x}, taking the triangle inequality and using \eqref{eq.fast} gives
\begin{equation*} 
\left|H(n+1)\right|\leq \left|x(n+1)\right|+F(\epsilon)\left|k\right|_1+\epsilon\sum_{j=0}^{n}\left|k(n-j)\right|\left|x(j)\right|. 
\end{equation*}
Letting $n\geq N_1$ with $N_1$ given by \eqref{eq.sumkj}, and $k_n(j)$ defined by \eqref{eq.knj}, we obtain
\begin{equation*}
\left|H(n+1)\right|\leq \left|x(n+1)\right|+F(\epsilon)\left|k\right|_1+\epsilon\left|k\right|_1 \sum_{j=0}^{n}k_n(j)\left|x(j)\right|.
\end{equation*}
Now, define $\theta_1=\epsilon\left|k\right|_1$, $\theta_2=1-2\epsilon\left|k\right|_1$, so $\theta_1, \theta_2\in (0,1)$ and $1-\theta_1-\theta_2=\epsilon\left|k\right|_1$.\\
Since $\varphi$ is an increasing function, by Jensen's Inequality (Lemma~\ref{lemma.Jensen}), we obtain
\begin{align*}
\lefteqn{
\varphi\left(\left|H(n+1)\right|\right)}
\\
&\leq \varphi\bigg(\left|x(n+1)\right|+F(\epsilon)\left|k\right|_1+\epsilon\left|k\right|_1 \sum_{j=0}^{n}k_n(j)\left|x(j)\right|\bigg)\\
&=\varphi\bigg(\theta_1\left(\frac{F(\epsilon)\left|k\right|_1}{\theta_1}\right)+\theta_2(1+\eta)\left|x(n+1)\right|+(1-\theta_1-\theta_2)\sum_{j=0}^{n}k_n(j)\left|x(n-j)\right|\bigg)\\
&\leq \theta_1\varphi\left(\frac{F(\epsilon)\left|k\right|_1}{\theta_1}\right)+\theta_2\varphi\big((1+\eta)\left|x(n+1)\right|\big)\\
&\qquad\qquad+(1-\theta_1-\theta_2)\varphi\bigg(\sum_{j=0}^{n}k_n(j)\left|x(n-j)\right|\bigg).
\end{align*}
Applying Lemma~\ref{lemma.Jensen} a second time gives
\begin{align*}
\lefteqn{\varphi\left(\left|H(n+1)\right|\right)}\\
&\leq \theta_1\varphi\left(\frac{F(\epsilon)\left|k\right|_1}{\theta_1}\right)+\theta_2\varphi\big((1+\eta)\left|x(n+1)\right|\big)+\epsilon\left|k\right|_1\sum_{j=0}^{n}k_n(j)\varphi\big(\left|x(n-j)\right|\big)\\
&=\theta_1\varphi\left(\frac{F(\epsilon)\left|k\right|_1}{\theta_1}\right)+\theta_2\varphi\big((1+\eta)\left|x(n+1)\right|\big)
\\
&\qquad \qquad
+\epsilon\left|k\right|_1\sum_{j=0}^{n} \left(\frac{\left|k(j)\right|}{\sum_{l=0}^{n}\left|k(l)\right|}\varphi\big(\left|x(n-j)\right|\big)\right).
\end{align*}
For $n\geq N_1$, by \eqref{eq.sumkj} we have 
$\sum_{j=0}^{n}\left|k(j)\right|\geq \frac{1}{2}\left|k\right|_1$. Therefore, for $n\geq N_1$ 
\begin{multline*}
\varphi\left(\left|H(n+1)\right|\right)\leq \theta_1\varphi\left(\frac{F(\epsilon)\left|k\right|_1}{\theta_1}\right)+\theta_2\varphi\big((1+\eta)\left|x(n+1)\right|\big)\\+2\epsilon\sum_{j=0}^{n}\left|k(j)\right|\varphi\big(\left|x(n-j)\right|\big).
\end{multline*}
Define 
\begin{equation*} 
D(\eta):=\theta_1\left(\epsilon(\eta)\right)\varphi\left(\frac{F(\epsilon)\left|k\right|_1}{\theta_1\epsilon(\eta)}\right)=\epsilon(\eta)\left|k\right|_1\varphi\left(\frac{F(\epsilon(\eta))}{\epsilon(\eta)}\right)
\end{equation*}
and
$\varphi_H^*:=\max_{{1}\leq{j}\leq{N_1+1}} \varphi\left(\left|H(n+1)\right|\right)$. 
Then for $0\leq n\leq N_1$, $\varphi\left(\left|H(n+1)\right|\right)\leq \varphi_H^*$ and for $n\geq N_1$ we obtain
\begin{multline*}
\varphi\left(\left|H(n+1)\right|\right)\leq \varphi_H^*+D(\eta)+\frac{1}{1+\eta}\varphi\big((1+\eta)\left|x(n+1)\right|\big)\\
+2\epsilon(\eta)\sum_{j=0}^{n}\left|k(j)\right|\varphi\left(\left|x(n-j)\right|\right). \end{multline*}
Hence for $n\geq 0$ 
\begin{multline*} \varphi\left(\left|H(n+1)\right|\right)\leq \varphi_H^*+D(\eta)+\frac{1}{1+\eta}\varphi\big((1+\eta)\left|x(n+1)\right|\big)\\+2\epsilon(\eta)\sum_{j=0}^{n}\left|k(j)\right|\varphi\left(\left|x(n-j)\right|\right). \end{multline*}
and thus for $N\geq 0$, we have
\begin{align*} 
\lefteqn{\sum_{n=1}^{N+1}\varphi\left(\left|H(n)\right|\right)}
\\
&=\sum_{n=0}^{N}\varphi\left(\left|H(n+1)\right|\right)\\
&\leq \big(\varphi_H^*+D(\eta)\big)(N+1)+\frac{1}{1+\eta}\sum_{n=0}^{N}\varphi\big((1+\eta)\left|x(n+1)\right|\big)\\
&\qquad+2\epsilon(\eta)\sum_{n=0}^{N}\sum_{j=0}^{n}\left|k(n-j)\right|\varphi(\left|x(j)\right|)\\
&=(\varphi_H^*+D(\eta))(N+1)+\frac{1}{1+\eta}\sum_{n=0}^{N}\varphi\big((1+\eta)\left|x(n+1)\right|\big)\\
&\qquad+2\epsilon(\eta)\sum_{j=0}^{N}\sum_{n=j}^{N}\left|k(n-j)\right|\varphi(\left|x(j)\right|)\\
&\leq\left(\varphi_H^*+D(\eta)\right)(N+1)+\frac{1}{1+\eta}\sum_{n=0}^{N}\varphi\big((1+\eta)\left|x(n+1)\right|\big)\\
&\qquad+2\epsilon(\eta)\left|k\right|_1\sum_{j=0}^{N}\varphi(\left|x(j)\right|)\\
\end{align*}
Hence
\begin{align*} 
\lefteqn{\sum_{n=1}^{N+1}\varphi\left(\left|H(n)\right|\right)}
\\
&\leq \left(\varphi_H^*+D(\eta)\right)(N+1)+\frac{1}{1+\eta}\sum_{n=1}^{N+1}\varphi\big((1+\eta)\left|x(n)\right|\big)\\
&\qquad+2\epsilon(\eta)\left|k\right|_1\sum_{n=0}^{N+1}\varphi\big((1+\eta)\left|x(n)\right|\big)\\
&\leq \left(\varphi_H^*+D(\eta)\right)(N+1)+\frac{1}{1+\eta}\sum_{n=1}^{N+1}\varphi\big((1+\eta)\left|x(n)\right|\big)\\
&\qquad+\left(1-\frac{1}{1+\eta}\right)\sum_{n=0}^{N+1}\varphi\big((1+\eta)\left|x(n)\right|\big).
\end{align*}
since $2\epsilon(\eta)\left|k\right|_1=1-\frac{1}{1+\eta}$.
Thus for $n\geq 1$ 
\begin{equation*} \sum_{j=1}^{n}\varphi\big(\left|h(j)\right|\big)\leq \left(\varphi_H^*+D(\eta)\right)n+\sum_{j=0}^{n}\varphi\big((1+\eta)\left|x(j)\right|\big). \end{equation*}
This yields
\begin{align*} \limsup_{n\to\infty}\frac{1}{n}\sum_{j=1}^{n}\varphi\left|h(j)\right|&\leq \varphi_H^*+D(\eta)+\limsup_{n\to\infty}\frac{1}{n}\sum_{j=0}^{n}\varphi\big((1+\eta)\left|x(j)\right|\big)<+\infty, \end{align*}
by hypothesis. This proves part (ii) of Theorem~\ref{theorem.ergodic}.


\end{document}